\documentclass[11pt,reqno,tbtags,a4paper]{amsart}
\usepackage{amssymb}
\usepackage{xpunctuate}
\usepackage{url}
\usepackage[square,numbers]{natbib}
\bibpunct[, ]{[}{]}{;}{n}{,}{,}

\title
{Tree limits and limits of random trees}

\date{28 May, 2020}

\author{Svante Janson}
\thanks{Supported by the Knut and Alice Wallenberg Foundation}
\address{Department of Mathematics, Uppsala University, PO Box 480,
SE-751~06 Uppsala, Sweden}
\email{svante.janson@math.uu.se}
\newcommand\urladdrx[1]{{\urladdr{\def~{{\tiny$\sim$}}#1}}}
\urladdrx{http://www2.math.uu.se/~svante/}


\numberwithin{equation}{section}

\renewcommand\le{\leqslant}
\renewcommand\ge{\geqslant}

\allowdisplaybreaks

\theoremstyle{plain}
\newtheorem{theorem}{Theorem}[section]
\newtheorem{lemma}[theorem]{Lemma}

\newtheorem{corollary}[theorem]{Corollary}

\theoremstyle{definition}

\newtheorem{exampleqqq}[theorem]{Example}
\newenvironment{example}{\begin{exampleqqq}}
  {\hfill\qedsymbol\end{exampleqqq}}

\newtheorem{remarkqqq}[theorem]{Remark}
\newenvironment{remark}{\begin{remarkqqq}}
  {\hfill\qedsymbol\end{remarkqqq}}

\newtheorem{definition}[theorem]{Definition}

\theoremstyle{remark}

\newenvironment{romenumerate}[1][-10pt]{
\addtolength{\leftmargini}{#1}\begin{enumerate}
 \renewcommand{\labelenumi}{\textup{(\roman{enumi})}}%
 }{\end{enumerate}}

\newenvironment{alphenumerate}[1][-10pt]{
\addtolength{\leftmargini}{#1}\begin{enumerate}
 \renewcommand{\labelenumi}{\textup{(\alph{enumi})}}%
 }{\end{enumerate}}

\newenvironment{PXenumerate}[1]{
\begin{enumerate}
 \renewcommand{\labelenumi}{\textup{(#1\arabic{enumi})}}%
 }{\end{enumerate}}

\newcounter{oldenumi}
{\setcounter{oldenumi}{\value{enumi}}
\begin{romenumerate} \setcounter{enumi}{\value{oldenumi}}}
{\end{romenumerate}}

\newcounter{thmenumerate}
\newenvironment{thmenumerate}
{\setcounter{thmenumerate}{0}%
 \def\item{\par
 \refstepcounter{thmenumerate}\textup{(\roman{thmenumerate})\enspace}}
}
{}

\newcounter{steps}
\newcommand\stepx{\smallskip\noindent\refstepcounter{steps}%
 \emph{Step \arabic{steps}. }\noindent}

\newcommand{\refT}[1]{Theorem~\ref{#1}}
\newcommand{\refTs}[1]{Theorems~\ref{#1}}
\newcommand{\refC}[1]{Corollary~\ref{#1}}

\newcommand{\refL}[1]{Lemma~\ref{#1}}
\newcommand{\refLs}[1]{Lemmas~\ref{#1}}
\newcommand{\refR}[1]{Remark~\ref{#1}}

\newcommand{\refS}[1]{Section~\ref{#1}}
\newcommand{\refSs}[1]{Sections~\ref{#1}}
\newcommand{\refSS}[1]{Section~\ref{#1}}

\newcommand{\refD}[1]{Definition~\ref{#1}}
\newcommand{\refDs}[1]{Definitions~\ref{#1}}
\newcommand{\refE}[1]{Example~\ref{#1}}
\newcommand{\refEs}[1]{Examples~\ref{#1}}

\newcommand\XREM[1]{\relax}

\begingroup
  \divide\count255 by 60
  \multiply\count255 by -60
  \advance\count255 by \time
  \ifnum \count255 < 10 \xdef\klockan{\the\count1.0\the\count255}
  \else\xdef\klockan{\the\count1.\the\count255}\fi
\endgroup

\newcommand\nopf{\qed}   

\newcommand{\sumko}{\sum_{k=0}^\infty}

\newcommand{\sumi}{\sum_{i=1}^\infty}
\newcommand{\sumj}{\sum_{j=1}^\infty}
\newcommand{\sumk}{\sum_{k=1}^\infty}
\newcommand{\suml}{\sum_{\ell=1}^\infty}

\newcommand{\sumiN}{\sum_{i=1}^N}

\newcommand{\sumkn}{\sum_{k=1}^n}

\newcommand\set[1]{\ensuremath{\{#1\}}}
\newcommand\bigset[1]{\ensuremath{\bigl\{#1\bigr\}}}

\newcommand\xpar[1]{(#1)}
\newcommand\bigpar[1]{\bigl(#1\bigr)}
\newcommand\Bigpar[1]{\Bigl(#1\Bigr)}

\newcommand\sqpar[1]{[#1]}
\newcommand\bigsqpar[1]{\bigl[#1\bigr]}
\newcommand\Bigsqpar[1]{\Bigl[#1\Bigr]}

\newcommand\xcpar[1]{\{#1\}}

\newcommand\bigabs[1]{\bigl\lvert#1\bigr\rvert}

\newcommand\lrabs[1]{\left\lvert#1\right\rvert}
\def\rompar(#1){\textup(#1\textup)}    
\newcommand\xfrac[2]{#1/#2}

\newcommand\Bigparfrac[2]{\Bigpar{\frac{#1}{#2}}}

\def\xexp(#1){e^{#1}}
\newcommand\ceil[1]{\lceil#1\rceil}

\newcommand\floor[1]{\lfloor#1\rfloor}

\newcommand\ntoo{\ensuremath{{n\to\infty}}}

\newcommand\ktoo{\ensuremath{{k\to\infty}}}

\newcommand\ttoo{\ensuremath{{t\to\infty}}}

\newcommand\bmin{\wedge}

\newcommand\punkt{\xperiod}    
\newcommand\iid{i.i.d\punkt}    
\newcommand\ie{i.e\punkt}
\newcommand\eg{e.g\punkt}

\newcommand\cf{cf\punkt}
\newcommand{\as}{a.s\punkt}
\newcommand{\textas}{\text{a.s.}}

\newcommand\whp{w.h.p\punkt}

\newcommand\ii{\mathrm{i}}

\newcommand{\tend}{\longrightarrow}
\newcommand\dto{\overset{\mathrm{d}}{\tend}}
\newcommand\pto{\overset{\mathrm{p}}{\tend}}
\newcommand\asto{\overset{\mathrm{a.s.}}{\tend}}
\newcommand\eqd{\overset{\mathrm{d}}{=}}

\newcommand\op{o_{\mathrm p}}

\newcommand\bbR{\mathbb R}
\newcommand\bbC{\mathbb C}
\newcommand\bbN{\mathbb N}

\newcommand\bbZ{\mathbb Z}

\newcounter{CC}
\newcommand{\CC}{\stepcounter{CC}\CCx} 
\newcommand{\CCx}{C_{\arabic{CC}}}     
\newcommand{\CCdef}[1]{\xdef#1{\CCx}}     
\newcommand{\CCname}[1]{\CC\CCdef{#1}}    
\newcounter{cc}

\newcommand\E{\operatorname{\mathbb E{}}}
\renewcommand\P{\operatorname{\mathbb P{}}}

\newcommand\Var{\operatorname{Var}}

\newcommand\Po{\operatorname{Po}}

\newcommand\Ge{\operatorname{Ge}}

\newcommand\diam{\operatorname{diam}}

\newcommand\ga{\alpha}
\newcommand\gb{\beta}
\newcommand\gd{\delta}
\newcommand\gD{\Delta}
\newcommand\gf{\varphi}
\newcommand\gam{\gamma}

\newcommand\kk{\kappa}
\newcommand\gl{\lambda}

\newcommand\go{\omega}

\newcommand\gs{\sigma}

\newcommand\gss{\sigma^2}
\newcommand\gth{\theta}
\newcommand\eps{\varepsilon}
\newcommand\ep{\varepsilon}

\newcommand\cA{\mathcal A}

\newcommand\cE{\mathcal E}
\newcommand\cF{\mathcal F}
\newcommand\cG{\mathcal G}

\newcommand\cK{\mathcal K}
\newcommand\cL{{\mathcal L}}

\newcommand\cN{\mathcal N}

\newcommand\cP{\mathcal P}

\newcommand\cS{{\mathcal S}}
\newcommand\cT{{\mathcal T}}

\newcommand\cV{\mathcal V}

\newcommand\cZ{{\mathcal Z}}

\newcommand\indic[1]{\boldsymbol1\xcpar{#1}}

\newcommand\qw{^{-1}}
\newcommand\qww{^{-2}}
\newcommand\qq{^{1/2}}

\newcommand\intoo{\int_0^\infty}

\newcommand\oi{\ensuremath{[0,1]}}

\newcommand\ooo{[0,\infty)}

\newcommand\ddx{\mathrm{d}}

\newcommand{\mgf}{moment generating function}

\newcommand{\gsf}{$\gs$-field}

\newcommand\lhs{left-hand side}
\newcommand\rhs{right-hand side}

\newcommand\GW{Galton--Watson}
\newcommand\GWt{\GW{} tree}
\newcommand\cGWt{conditioned \GW{} tree}
\newcommand\GWp{\GW{} process}

\newcommand\xoo{_1^\infty}

\newcommand\Moo{M_\infty}
\newcommand\rhooo{\rho_\infty}
\newcommand\tauoo{\tau_\infty}
\newcommand\push[2]{#1(#2)}
\newcommand\nn{^{(n)}}
\newcommand\nnx[1]{^{(#1)}}
\newcommand\nnga{\nnx{\ga}}
\newcommand\xnn{_n}
\newcommand\xnu{\gl}
\newcommand\cpmoo{\cP(\Moo)}
\newcommand\Ce{C_\eps}
\newcommand\Q[1]{\mathbf{#1}}
\newcommand\QT{\Q{T}}
\newcommand\Qd{\Q{d}}
\newcommand\Qmu{\Q{\mu}}
\newcommand\Qx{\Q{x}}

\newcommand\fD{\mathfrak D}
\newcommand\fT{\mathfrak T}
\newcommand\rT{\hat T}
\newcommand\fTf{\mathfrak T_{\mathsf f}}
\newcommand\fTr{\mathfrak T_{\mathsf r}}
\newcommand\cTf{\mathcal T_{\mathsf f}}
\newcommand\cTr{\mathcal T_{\mathsf r}}
\newcommand\cTc{\cT_{\mathsf c}}
\newcommand\OO{\Upsilon}
\newcommand\ET{Elek--Tardos}
\newcommand\ETn{Elek--Tardos normalization}
\newcommand\Bex{{\mathbf e}}
\newcommand\Tx{T}
\newcommand\cTex{\Tx_{2\Bex}}
\newcommand\BCRT{Brownian continuum random tree}

\newcommand\GHP{Gromov--Hausdorff--Prohorov}
\newcommand\GP{Gromov--Prohorov}
\newcommand\ddd{\gd} 
\newcommand\DDD{\gD} 
\newcommand\ctn{\cT_n}
\newcommand\ctt{\ctx{t}}
\newcommand\ctx[1]{\tilde\cT_{#1}}
\newcommand\vvv{v^\dagger} 

\newcommand\ttt{{\mathbf{t}}}
\newcommand\httt{{\mathbf{\hat t}}}
\newcommand\TT{\mathbf{T}}
\newcommand\TTx{\mathbf{T'}}
\newcommand\hv{\hat v}
\newcommand\Ezeta{\kk}
\newcommand\xW{\overline W}
\newcommand\TTK{\TT_K}
\newcommand\TTKK{\TT_{>K}}

\newcommand\bbNo{\bbN_0}
\newcommand\Tbx{\mathrm{T}}
\newcommand\Tb{\Tbx_b}
\newcommand\Too{\Tbx_{\infty}}
\newcommand\sumiob{\sum_{i=0}^b}
\newcommand\sumib{\sum_{i=1}^b}
\newcommand\hN{\widehat N}
\newcommand\VV{\widehat V}
\newcommand\cEp{\cE_+}
\newcommand\cEm{\cE_-}

\newcommand\ZZB{\cZ^\mathsf{b}}
\newcommand\ZZV{\cZ}

\newcommand\ZZVL{\ZZV_{\le L}}
\newcommand\dx{d}
\newcommand\hmu{\widehat\mu}
\newcommand\xix{\bar\xi}
\newcommand\Xia{\widehat\Xi(\ga)}
\newcommand\bb{b}
\newcommand\YY{Y}

\newcommand{\Lovasz}{Lov\'asz}

\begin{document}

\begin{abstract} 
We explore the tree limits recently defined by Elek and Tardos.
In particular, we find tree limits for many classes of random trees.
We give general theorems for three classes of
conditional Galton--Watson trees and simply generated trees, 
for split trees and generalized split trees (as defined here),
and for trees defined by a continuous-time branching process.
These general results include, for example, 
random labelled trees, ordered trees,
random recursive trees, 
preferential attachment trees, 
and
binary search trees.
\end{abstract}

\maketitle

\section{Introduction}\label{S:intro}

\citet{ET} have recently introduced a theory of tree limits, in
analogy with the theory of graph limits \cite{Lovasz} and other similar
limits of various combinatorial objects 
(\eg{} hypergraphs, permutations, \dots).
Their idea is to regard a tree as a metric space with a probability measure;
the metric is the usual graph distance, suitably rescaled, and the probability
measure is the uniform  measure on the vertices.
Then,  for each integer $r$, consider the
random matrix $(d(\xi_i,\xi_j))_{i,j=1}^r$ of distances between $r$ random
vertices 
$\xi_1,\dots,\xi_r$. A sequence of trees is said to converge if, for each
$r\ge1$, the resulting random $r\times r$ matrices converge in distribution.
(This type of convergence  for metric spaces with a measure 
goes back to \citet[Chapter $3\frac12$]{Gromov}.)
See
\refS{SET} for details of this and of other topics mentioned below. 

\citet{ET} choose to normalize the metrics of the trees by dividing
the graph distance by the diameter; hence the trees become metric spaces
with diameter 1. 
One reason for this normalization is that this embeds the trees in a compact
space, and thus every sequence of trees has a convergent subsequence.
However, the theory developed in \cite{ET} treats also
more general real trees, and include trees with a different normalization.
We will in general not  use the \ETn, since other scalings
often seem more natural, in particular for random trees,
see \eg{} \refEs{ESn} and \ref{EBn}, and
\refSs{SGW1} and \ref{SGW2}.

The main results by \citet{ET} are that there exists 
a set of limit objects
called \emph{dendrons} such that each convergent sequence of finite trees
(with their normalization) converges to a unique dendron.
The dendrons can be regarded as real trees equipped with probability
measures, 
but the precise definition is slightly different.
The dendrons are defined as special cases of \emph{long dendrons}; dendrons
are long dendrons with diameter at most 1. This is tied to the
\ETn, and in the present paper, the main limit objects
are the long dendrons.

In some cases, the (long) dendrons can be identified with real trees, and
the tree limits then  coincide with 
limits in the \GP{} metric
(see \refR{RGP}). Such limits have been studied earlier (also in
the stronger \GHP{} metric); one much studied example 
going back to \citet{AldousI,AldousII,AldousIII}
is provided by \cGWt{s}, 
see \refS{SGW1}.
However the \ET{} limits are more general, and include also other types of
limits; one
 example is provided by a different class of \cGWt{s} (where condensation
 appears in the limit),  see \refS{SGW2}.

The tree limits by \citet{ET} thus seem to be very interesting, and
promising for future research.
The purpose of the present paper is to further develop the theory of tree
limits.
We give some general results in \refSs{SET}--\ref{Scompact}.
In particular, 
we show how  the set of all tree limits, or equivalently the set
of all long dendrons, can be regarded as a  metric (and Polish) space;
this makes it possible to define and study random tree limits and limits of
random trees in a convenient way.
Moreover, we characterize relative
compactness of a sequence (or set) of rescaled trees
(\refT{TC} and \refC{CC}).

The second, and perhaps main, part of the paper
applies the general theory to several classes of random trees
and finds tree limits for them.
As a preparation, we give in \refS{Sex} some simple examples of limits of
deterministic trees.
A few general results on limits of random trees are given in \refS{SRT}.

The following sections study first different classes of 
\cGWt{} and simply generated trees
(\refSs{SGW1}--\ref{SIII}),
and then different classes of random trees with logarithmic height
(\refSs{Slog}--\ref{SCMJ}), in particular
split trees (\refS{Ssplit}) 
and trees defined by continuous time branching processes (\refS{SCMJ}).
We find tree limits in all these cases, as the size \ntoo; in some cases 
with convergence in distribution to a random tree limit, and in others with
convergence in probability to a fixed tree limit.

The found limits are of different types. In particular, for a class of
\cGWt{s} including many standard classes of random trees with height of
order $\sqrt n$ (\refS{SGW1}), the well-known limit theorem by
\citet{AldousIII} gives convergence to a random real tree known as the
  the \BCRT; this tree can be regarded as a (random) long dendron, and 
Aldous's result holds in the present sense too.

On the other hand, many standard classes of random trees with height of
order $\log n$ are covered by the general results in
\refSs{Slog}--\ref{SCMJ} and have tree limits of a quite different type;
these limits are long dendrons of a very simple type (but distinct from real
trees), which is equivalent to the fact that in these trees, almost all
pairs of vertices have almost the same distance. (This is shown more
generally in \refT{TOOX}.)

\begin{remark}
  The tree limits by \citet{ET} studied in the present paper are global
  limits, in general quite different from local limits studied in \eg{} 
  \cite{SJ264}. Nevertheless, there are cases (see for example \refS{SGW2})
  where the trees are such that there is a strong relation between 
the tree limits and local limits.
\end{remark}

\section{Some notation}\label{Sprel}

\subsection*{Probability measures}

Recall that a Polish space is a separable completely metrizable topological
space. In other words, it can be regarded as a complete separable metric
space, but we ignore the metric. (When necessary or convenient, we can
choose a metric, but there is no distinguished one.)
A Polish space is often regarded as a measurable space, equipped with its
Borel \gsf.

If $X=(X,\cF)$ is a  measurable space, 
then $\cP(X)$ is the space of probability measures on $X$.
In particular,
if $X$
is a metric space, then $\cP(X)$ is the space of Borel
probability measures on $X$, and in this case we
equip $\cP(X)$ with the standard weak topology, see
\eg{} \cite{Billingsley}.
If $X$ is a Polish space, then so is $\cP(X)$,
see 
\cite[Appendix III]{Billingsley} or
\cite[Theorem 8.9.5]{Bogachev}.

The Dirac measure (unit point mass)  at a point $x$ is denoted $\gd_x$.

If $\mu$ is a probability measure on a space $X$, then $\xi\sim\mu$ and
$\mu=\cL(\xi)$ both denote that $\xi$ is a random element of $X$ with
distribution $\mu$.

If $X=(X,\cF)$ and $Y=(Y,\cG)$ are measurable spaces, 
$\gf:X\to Y$ is a measurable map, and 
$\mu\in\cP(X)$, then the \emph{push-forward} 
$\push{\gf}{\mu}\in\cP(Y)$ of $\mu$
is defined by
\begin{align}\label{push}
\push{\gf}{\mu}(A):=\mu\bigpar{\gf\qw(A)},\qquad A\in\cG.
\end{align}
(This is often denoted $\mu\circ\gf\qw$
or $\gf_*(\mu)$.)
Equivalently, if $\xi$ is a random element of $X$  
then 
\begin{align}\label{push2}
\xi\sim\mu 
\implies
\gf(\xi)\sim\push{\gf}{\mu}.  
\end{align}

\subsection*{Limits}

Unspecified limits are as \ntoo.

As usual, {\whp} (\emph{with high probability})
means with probability tending to 1 as a parameter (here always $n$) tends
to $\infty$.

If $Z,Z_n$ are random elements of a metric space $X$, then 
$Z_n\dto Z$, $Z_n\pto Z$, and $Z_n\asto Z$ 
denote convergence in distribution, in probability and almost surely (\as),
respectively. 
Note that $Z_n\dto Z$  is the same as
convergence in $\cP(X)$ of the distributions, \ie, $\cL(Z_n)\to \cL(Z)$.

If $(a_n)_n$ is a sequence of positive numbers, then $\op(a_n)$ denotes a
sequence of random variables  $Z_n$ such that $Z_n/a_n\pto0$; 
this is equivalent to $|Z_n|/a_n<\eps$ \whp{} for every $\eps>0$.

\subsection*{Miscellaneous}

If $T$ is a tree, we abuse notation and write $T$ for its vertex set $V(T)$.
The number of vertices is denoted by $|T|$.
If $T$ is a rooted tree, then the root is denoted by $o$.

If $x,y\in\bbR$, then $x\bmin y:=\min\set{x,y}$.
On the other hand, if $v$ and $w$ are vertices in a rooted tree, then
$v\bmin w$ denotes their last common ancestor.

For a sequence of random variables, \iid{} means  
independent and identically distributed.

$\bbN:=\set{1,2,\dots}$ and $\bbNo:=\set{0,1,2,\dots}$.

$C$ and $c$ denote positive constants; these may vary from one occurrence to
another. (We sometimes distinguish them by subscripts.)

\section{Convergence of trees and long dendrons}\label{SET}

We give here a summary of the main definitions and results of \cite{ET},
together with some further notation.

\subsection{Convergence of trees}\label{SStrees}
For $r\ge1$, let $M_r$ be the space of real $r\times r$ matrices;
note that $M_r=\bbR^{r^2}$ is a Polish space, and thus
$\cP(M_r)$ is a Polish space.

For a set $X$ with a given function $d:X^2\to\bbR$, and $r\ge1$, let
$\rho_r:X^r\to M_r$ be the map given by the entries
\begin{align}\label{rhor}
  \rho_r(x_1,\dots,x_r)_{ij}=
  \rho_r(x_1,\dots,x_r;X,d)_{ij}:=
  \begin{cases}
    d(x_i,x_j),& i\neq j,
\\
0, & i=j.
  \end{cases}
\end{align}
We often consider $\rho_r$ when $d$ is a metric on $X$; then the special
definition in \eqref{rhor}
when $i=j$ is redundant. However, for the long dendrons defined below, we 
typically have $d(x,x)>0$, and then the definition \eqref{rhor} is important.
See also \refR{Rd-t}.

A \emph{metric measure space} is a triple $(X,d,\mu)$, where $X$ is a
measurable space (so $X=(X,\cF)$ with $\cF$ hidden in the notation),
$\mu\in\cP(X)$, and $d:X^2\to\bbR$ is a measurable metric on $X$. 

Suppose, more generally, that
$X=(X,\cF,\mu)$ is a probability space and that $d:X^2\to\bbR$ is a
measurable function.  For $r\ge1$, define
the sampling measure
\begin{align}\label{taur}
  \tau_r(X)=\tau_r(X,d,\mu):=\push{\rho_r}{\mu^r}\in \cP(M_r),
\end{align}
the push-forward of the measure
$\mu^r\in \cP(X^r)$ along $\rho_r$. In other words,
if $\xi_1,\dots,\xi_r$ are 
\iid{}
 random points in $X$ with $\xi_i\sim\mu$,
then
\begin{align}\label{taurus}
  \tau_r(X):=\cL\bigpar{\rho_r(\xi_1,\dots,\xi_r;X)}, 
\end{align}
the distribution of  the random matrix 
$\rho_r(\xi_1,\dots,\xi_r)\in M_r$.

A finite tree $T$ is regarded as a metric space $(T,d_T)$, where $d_T$ is the
graph distance. Furthermore, if $c>0$, we let $cT$ denote the metric space
$(T,cd_T)$, where all distances are rescaled by $c$.
We regard $cT$ as a metric probability space by equipping it with the uniform
measure $\mu_T$ defined by $\mu_T\set{x}=1/|T|$ for $x\in T$.
Then $\tau_r(cT)\in\cP(M_r)$ is defined by \eqref{taur}.

\begin{definition}  \label{D1}
Let $(T_n)\xoo$ be a sequence of finite trees and $(c_n)\xoo$ a sequence of
positive numbers.
Then the sequence $(c_nT_n)\xoo$ converges if the sampling measures
converge for every fixed $r$, \ie, if there exist $\xnu_r\in \cP(M_r)$ such
that,
as \ntoo,
\begin{align}\label{d1}
\tau_r(c_nT_n)=\tau_r\bigpar{T_n,c_nd_{T_n},\mu_{T_n}}\to \xnu_r  
\qquad\text{in $\cP(M_r)$, $r\ge1$}. 
\end{align}
\end{definition}
By \eqref{taurus}, the condition \eqref{d1} is equivalent to convergence in
distribution of the random matrices
$\rho_r(\xi\nn_1,\dots,\xi\nn_r;c_nT_n)$, 
where for each $n$, $(\xi\nn_i)_i$
are \iid{} uniform random vertices of $T_n$.

\begin{remark}\label{Rdiam}
  As said in the introduction, \citet{ET} consider only the
  normalization $c_n=1/\diam(T_n)$, but we will not assume this.
\end{remark}

A \emph{real tree} is a 
complete
non-empty metric space $(T,d)$ such that
for any pair of distinct points $x,y\in T$, there exists a unique isometric
map $\ga:[0,d(x,y)]\to T$ with $\ga(0)=x$ and $\ga(d(x,y))=y$, 
and furthermore, for every $s\in(0,d(x,y))$, $x$ and $y$ are in different
components of $T\setminus\set{\ga(s)}$.
(There are several different but equivalent versions of the definition; see
\eg{} \cite{Dress,DressMoultonT,LeGall2005,LeGall2006}.)

\begin{remark}
  Note that we define the trees as complete (as do \cite{ET}); this is often
  not required. For our purposes completeness is convenient and no real loss of
    generality; if $T$ is an incomplete real tree
(defined as above without completeness), then the completion $\overline T$
is also a real tree (see \eg{} \cite[Theorem 8]{Dress}), and 
in the limit theory below we can use $\overline T$ instead of $T$.
\end{remark}

If $T=(T,d)$ is a real tree and $c>0$, let $cT:=(T,cd)$. Then $cT$
is also a real tree. 

A \emph{measured real tree} is a real tree $T=(T,d)$
equipped with a probability measure $\mu$. 
We will only consider separable trees $T$ and 
Borel measures $\mu$, and then 
$(T,d,\mu)$ is always a metric measure space. 
(For non-separable measured real trees, see \cite{ET},
where they \eg{} are  used in the proofs;
then $\mu$ might be defined on a smaller \gsf{} than the Borel one, and  the
condition that $d$ has to be measurable is added.
See also  \refR{Rd-t}.)

\begin{example}\label{Ereal}
If $T$ is any finite tree (in the usual combinatorial sense), let $\rT$
denote the 
real tree obtained by regarding each edge in $T$ is an interval of length 1.
Then $\rT$ is a compact real tree, and $T$ is isometrically embedded as a
subset of $\rT$.
Hence, we can regard $\mu_T$ as a probability measure on
$\rT$, and $(\rT,\mu_T)=(\rT,d,\mu_T)$ is a measured real tree.
Obviously,
$\tau_r(T)=\tau_r(\rT,\mu_T)$.
 More generally, $cT$ is isometrically embedded in $c\rT$
for any $c>0$, and
\begin{align}\label{ctt}
 \tau_r(cT)=\tau_r(c\rT,\mu_T). 
\end{align}
We can therefore sometimes identify $\rT$ and $T$;
see \refS{Sabs}.
\end{example}

Consequently, we can regard \refD{D1} as a special case of the following
definition. 

\begin{definition}  \label{D2}
Let $(T_n)\xoo=(T_n,d_n,\mu_n)\xoo$ 
be a sequence of measured real trees.
Then the sequence $(T_n)\xoo$ converges if the sampling measures
converge for every fixed $r$, \ie, if there exist $\xnu_r\in \cP(M_r)$ such
that,
as \ntoo,
\begin{align}\label{d2}
\tau_r(T_n)  
\to \xnu_r  
\qquad\text{in $\cP(M_r)$, $r\ge1$}. 
\end{align}
Again, \eqref{d2} 
is equivalent to convergence in
distribution of 
the random matrices
$\rho_r(\xi\nn_1,\dots,\xi\nn_r;T_n)$, where, for each $n$, 
$(\xi_i\nn)_i$ are \iid{}  random points in $T_n$
with  $\xi\nn_i\sim\mu_n$.
\end{definition}

\begin{remark}
  \label{RGP}
\citet[Chapter $3\frac12$]{Gromov} studied general complete separable metric
measure spaces (with a finite Borel measure, which we may normalize to be a
probability measure as above). He defined the \GP{} metric 
(see \citet[p.~762]{Villani} for another version), and he also considered
convergence in the sense above, \ie, $\tau_r(X_n)\to\tau_r(X)$ for every
$r$, where $X_n$ and $X$ are metric measure spaces; it turns out that this
is equivalent to convergence in the \GP{} metric, see \cite{SJN22}.
\citet[$3\frac12$.14 and $3\frac12$.18]{Gromov} noted also that
it is possible that $\tau_r(X_n)$ converges for every $r$ to some limit
measure, but that there is no metric measure space $X$ that is the limit.
(One of Gromov's examples is the sequence of unit spheres 
$S^n$ with uniform measure,
which behave as in \refT{TOO} below with almost all distances being almost
equal; a metric measure space limit would have to have almost all distances
equal to some positive constant, which is impossible for separable spaces.)
The new idea by \citet{ET} is to define another type of limit object (long
dendrons) that works in general when $X_n$ are trees.
\end{remark}

\subsection{Long dendrons}\label{SSdendron}
\citet{ET} defined  limit objects 
as follows.
Note that a real tree $T$ is locally connected (and locally pathwise connected);
thus, if $p\in T$, then $T\setminus\set{p}$ is the disjoint union of one or
several (possibly infinitely many) open connected components; these are
called \emph{$p$-branches}. A \emph{branch} of $T$ is a $p$-branch for some
$p\in T$.

\begin{definition}\label{DD}
  A \emph{long dendron} $D=(T,d,\nu)$ is a real tree $(T,d)$ together with a
  (Borel) probability measure $\nu$ on $A_D:=T\times\ooo$
satisfying $\nu(B\times\ooo)>0$ for every branch $B$ of $T$.
We define $d_D:A_D^2\to\ooo$ by
\begin{align}\label{dD}
  d_D\bigpar{(x,a),(y,b)}:=d(x,y)+a+b.
\end{align}

An isomorphism between two long dendrons $D=(T,d,\nu)$ and $D'=(T',d',\nu')$
is an isometry $f$ from $(T,d)$ onto $(T',d')$ such that the mapping 
$\bar f:=(p,a)\mapsto(f(p),a)$ is measure-preserving 
$(A_D,\nu)\to (A_{D'},\nu')$. 
\end{definition}

We call the real tree $T$ the \emph{base} of the long dendron $D$; 
we may identify
$T$ with $T\times\set0\subset A_D$.
It is shown in \cite[Lemma 6.2]{ET} that the base $T$ of a dendron
necessarily is separable; thus $T$ and $A_D$ are Polish spaces.

\begin{remark}\label{Rshort}
  \citet{ET} also define a \emph{dendron} as a long dendron such that 
if $\xi_1,\xi_2\in A_D$ are \iid{} random points with distribution $\nu$,
then $d_D(\xi_1,\xi_2)\le1$ a.s.
These are the limit objects for real trees with diameter $\le1$, and thus
for trees with the \ETn{} in \refR{Rdiam}, but they
have no special importance in the present paper.
We may call them \emph{short dendrons}.
(For consistency with \cite{ET}, we keep the name long dendron, although
for our purposes it would be more natural to change
notation and call them dendrons.)
\end{remark}

For a long dendron $D=(T,d,\nu)$, we use again \eqref{taur}--\eqref{taurus}
and define
the sampling measure
\begin{align}\label{bob}
  \tau_r(D):=\tau_r(A_D,d_D,\nu)\in\cP(M_r),
\end{align}
\ie, the distribution of  the random matrix 
$\rho_r(\xi_1,\dots,\xi_r; A_D,d_D)\in M_r$ if $\xi_1,\dots,\xi_r$ are 
\iid{}
random points in $A_D$ with $\xi_i\sim\nu$.

Convergence of finite or real trees
to a long dendron is defined by adding to \refDs{D1} and
\ref{D2} that the limits of the sampling measures are the sampling measures 
for the limit:
 \begin{definition}\label{D3}
 Let $(T_n)\xoo$ be a sequence of finite trees and $(c_n)\xoo$ a sequence of
positive numbers, and let $D$ be a long dendron.
Then the sequence $(c_nT_n)\xoo$ converges to $D$
if,
as \ntoo,
\begin{align}\label{d3a}
\tau_r(c_nT_n)\to\tau_r(D)
\qquad\text{in $\cP(M_r)$, $r\ge1$}. 
\end{align}
Similarly, if
$(T_n)\xoo$ is a sequence of real trees and  $D$  a long dendron,
then  $T_n$ converges to $D$
if,
as \ntoo,
\begin{align}\label{d3b}
\tau_r(T_n)\to\tau_r(D)
\qquad\text{in $\cP(M_r)$, $r\ge1$}. 
\end{align}
 \end{definition}

Again, \eqref{d3a} and \eqref{d3b} are equivalent to convergence in
distribution of the random matrices
$\rho_r(\xi\nn_1,\dots,\xi\nn_r)$, 
where $(\xi\nn_i)_i$ are \iid{} as above.

\begin{remark}\label{Rd-t}
As discussed in \cite[Remark 4]{ET},
the long dendrons could be replaced by real trees as follows.
(We might think of long dendrons as proxies for some measured real trees.)
Let $D=(T,d,\nu)$ be a long dendron.
First, if $\nu\set{t}=0$ for every $t\in T$, 
consider $A_D=T\times\ooo$ as a real tree $T'$ consisting of
$T=T\times\set0$ 
with a half-line
$\set{t}\times\ooo$ attached at each $t\in T$.
In general, we have to attach a continuum of half-lines at each point $t$ 
(so that each half-line has measure 0), for example by defining
$T':=T\times\bbC$
regarded as $T$ with the half-lines
$\set{(t,re^{\ii\gth}):r\ge0}$ attached,
for every $t\in T$ and $\gth\in[0,2\pi)$,
and with the measure $\nu'$ on $T'$ equal to the push-forward 
by the map $(t,r,\gth)\mapsto(t,re^{\ii\gth})$ of the measure
$\nu\times\ddx\gth/2\pi$. 
Note that then $\tau_k(T')=\tau_k(D)$ for every $k\ge1$.
(Note how the special definition for $i=j$ in \eqref{rhor} interacts with
\eqref{dD} to give the desired result.)

However, we agree with \citet{ET} that it is more convenient to use long
dendrons as limit objects.
One reason is that the trees just constructed are nonseparable, and that the
measures are not Borel measures on $T'$. (There are plenty of nonmeasurable
open sets.) Another reason is that long dendrons provide uniqueness of the
limits in a simple way.
\end{remark}

\subsection{Two examples}\label{SS2ex}
The following examples of long dendrons are rather simple, and extreme in
the sense that the measure $\nu$ on $A_D=T\times \ooo$ is 
supported on a
'one-dimensional' set with one of the coordinates fixed. 
Nevertheless, these two examples will play the main role in our  limit
theorems for random trees.

\begin{example}\label{ED=T}
  Let $T=(T,d,\mu)$ be a 
measured real tree such that every branch has
positive measure.
Identify $T\times\set0$ with $T$, and define $\nu$ as $\mu$ regarded as a
measure on $T\times\set0\subset A_D:=T\times\ooo$. 
(More formally, $\nu$ is
the push-forward of $\mu$ under $x\mapsto(x,0)\in A_D$.)
Then $(T,d,\nu)$ is a long dendron. 

Note that if $\xi\sim\mu$, then $(\xi,0)\sim \nu$. 
Since \eqref{dD} implies that $d_D$
equals $d$ on $T\times\set0=T$, 
it follows that $\tau_r(D)=\tau_r(T)$ for every $r\ge1$.
We may thus identify the long dendron $D$ with the measured real tree $T$.

By \refR{RGP}, convergence of a sequence of (real) trees
to $D$ as in \refD{D3} is equivalent to
convergence to $T$ in the \GP{} metric.
\end{example}

\begin{example}\label{EOO}
  Let $T=\set \bullet$ be the real tree consisting of a single point. 
Then the metric $d=0$, 
and we let $\mu=\gd_\bullet$  (the only probability measure on $T$, so there
is no choice).

We may identify $A_D=\set{\bullet}\times\ooo$ with $\ooo$. Thus every
probability measure $\nu$ on $\ooo$ defines a long dendron
$\OO_\nu:=(T,d,\nu)$.

By \eqref{rhor} and \eqref{dD}, 
\begin{align}\label{pop}
\rho_r(\xi_1,\dots,\xi_r;\OO_\nu)
=\bigpar{(\xi_i+\xi_j)\indic{i\neq  j}}_{i,j=1}^r,
\end{align}
and thus $\tau_r(\OO_\nu)$ is the distribution of the matrix \eqref{pop}
when $(\xi_i)_i$ are \iid{} with $\xi_i\sim\nu$.

A particularly simple, and important, case is when $\nu=\gd_a$ for some
$a\ge0$. In this case we denote the long dendron by $\OO_a$, and note that
$\xi_i=a$ is non-random, and thus \eqref{pop} shows that
$\rho_r(\xi_1,\dots\xi_r)$ is the constant matrix
\begin{align}\label{mom}
\rho_r(\xi_1,\dots,\xi_r;\OO_a)
=\bigpar{2a\indic{i\neq  j}}_{i,j=1}^r.
\end{align}
This leads to the following simple characterization of convergence to the
long dendron $\OO_a$.
\end{example}

\begin{theorem}\label{TOO}
Let  $(c_nT_n)_n$ be a sequence of rescaled trees, and $a\ge0$.
Then $c_nT_n\to \OO_a$ if and only if
\begin{align}\label{too}
  c_nd_n(\xi\nn_1,\xi\nn_2)\pto 2a,
\end{align}
where $d_n$ is the graph distance in $T_n$ and
$(\xi\nn_i)_i$ are \iid{}
uniformly random vertices in $T_n$.

The same holds, mutatis mutandis, for a sequence $(T_n,d_n,\mu_n)$ of
measured real trees.  
\end{theorem}

In the terminology of \citet[p.~142]{Gromov},
\eqref{too} says that $c_nT_n$ have  
\emph{(asymptotic) characteristic size} $2a$.

\begin{proof}
Convergence in distribution to a constant is the same as convergence in
probability. Thus, \eqref{taurus}, \eqref{rhor} and \eqref{mom} 
show that \refD{D3} now yields
\begin{align}\label{sw}
& c_nT_n\to\OO_a
\notag\\&\qquad
\iff 
\bigpar{c_nd_n(\xi\nn_i,\xi\nn_j)\indic{i\neq  j}}_{i,j=1}^r
\pto\bigpar{2a\indic{i\neq  j}}_{i,j=1}^r,
\quad r\ge1,
\notag\\&\qquad
\iff 
{c_nd_n(\xi\nn_i,\xi\nn_j)\indic{i\neq  j}}
\pto{2a\indic{i\neq  j}},
\quad i,j\ge1.
\end{align}
By symmetry, it suffices to consider the case $i=1$, $j=2$.
\end{proof}

\begin{remark}\label{ROO0}
  The (long) dendron $\OO_0$ is trivial, with $d_D(\xi_1,\xi_2)=0$ \as{} if
  $\xi_i\sim\nu=\gd_0$. Note that $\OO_0$ equals the equally trivial
real tree $T=\set \bullet$ 
consisting of a single point, regarded as a long dendron as in \refE{ED=T}.

The  trivial long dendron $\OO_0$ is by \refT{TOO} the limit 
of $c_nT_n$ when
\begin{align}
  \label{trivial}
c_nd_n(\xi\nn_1,\xi\nn_2)\pto0,
\end{align}
which typically means that we have chosen the wrong rescaling.
\end{remark}

\subsection{Limit theorems}\label{SSlimit}
Some of the main results of \citet{ET} are the following, 
here somewhat reformulated.

\begin{theorem}[{\cite[partly Theorems 1 and 4]{ET}}]\label{ET1}
Any convergent sequence of rescaled finite trees converges to some long dendron.
The same holds for any convergent sequence of measured real trees.
\end{theorem}

This is not stated in quite this generality in \cite{ET}; we show in
\refS{SpfET1} 
how it follows from other results in \cite{ET}.
(We postpone this proof until the end of the paper because it uses arguments
from \cite{ET} quite different from the other arguments in the present paper.)

\begin{theorem}[{\cite[Theorem 2, Lemmas 7.1 and 7.2]{ET}}]\label{ET2}
  Any long dendron is the limit of a convergent sequence $(c_nT_n)\xoo$
of rescaled finite trees.
\nopf
\end{theorem}
Again, this is not stated in quite this form in \cite{ET}, but it is a
simple consequence of \cite[Lemmas 7.1 and 7.2]{ET}; we omit the details.

\begin{theorem}[{\cite[Theorem 3]{ET}}]\label{ET3} %
  Two long dendrons $D$ and $D'$ are isomorphic if and only if
  $\tau_r(D)=\tau_r(D')$ for every $r\ge1$.
Consequently, the limit of a sequence of real trees (or rescaled finite trees)
is unique (up to isomorphism) if it exists.
\nopf
\end{theorem}

\section{Infinite matrices}\label{Soo}

We extend the definitions in \refS{SET} to the case $r=\infty$, \ie{} to  
infinite matrices.
Let $\Moo$ be the space of infinite real matrices $(a_{ij})_{i,j=1}^\infty$.
Define $\rho_r$ and $\tau_r$ by \eqref{rhor} and
\eqref{taur}--\eqref{taurus} 
also for
$r=\infty$; thus $\tauoo(X,d,\mu)$ is the distribution of the infinite random
matrix $\bigpar{d(\xi_i,\xi_j)\indic{i\neq j}}_{i,j\ge1}$ where $\xi_i$ are
\iid{} with  $\xi_i\sim\mu$.

Given any $A=(a_{ij})_{i,j=1}^s\in M_s$, with $r\le s\le \infty$,  
define the restriction 
\begin{align}
  \Pi_r(A)=(a_{ij})_{i,j=1}^r\in M_r, 
\end{align}
\ie,
the $r\times r$ top left corner of $A$.
Furthermore, if $A\in M_s$ is a random matrix with distribution
$\nu\in \cP(M_s)$, we denote the distribution of $\Pi_r(A)$ by
$\push{\Pi_r}{\nu}\in\cP(M_r)$. (This is the push-forward of $\nu$, see
\eqref{push}.)
In other words, $\Pi_r(\nu)$ is the marginal distribution of
the $r\times r$ top left corner.

Say that a sequence $\xnu_r\in\cP(M_r)$, $1\le r<\infty$, is \emph{consistent}
if $\Pi_r(\xnu_s)=\xnu_r$ when $r\le s$. 

If $\xnu\in\cP(\Moo)$, then the sequence $\xnu_r:=\Pi_r(\xnu)$ is obviously
consistent. Conversely, every consistent sequence arises in this way for a
unique $\xnu\in\cP(\Moo)$;
the corresponding statement for distributions of random vectors in
$\bbR^\infty$ is well-known  
\cite[Theorem 6.14]{Kallenberg}, and the result for $\Moo$ 
follows immediately by reading the entries of the matrices in a suitable
fixed order.
Furthermore, if $\xnu,\xnu\xnn\in\Moo$, then
\begin{align}\label{erika}
  \xnu\xnn\to\xnu \text{ in $\cP(\Moo)$}
\iff
\Pi_r(\xnu\xnn)\to\Pi_r(\xnu)
 \text{ in $\cP(M_r)$ for each $r\ge1$}.
\end{align}
Again, this follows immediately from the corresponding well-known fact for
$\bbR^\infty$ \cite[p.~19]{Billingsley}.

A sequence $\tau_r(X)$, $r\ge1$, given by \eqref{taurus} is obviously
consistent; furthermore, $\tau_r(X)=\Pi_r(\tauoo(X))$ for every $r$.
Consequently, \eqref{erika} implies the following.

\begin{theorem}\label{Too}
Let $(T_n)\xoo=(T_n,d_n,\mu_n)\xoo$ be a sequence of measured real trees,
and let $D=(T,d,\nu)$ be a long dendron.
  \begin{romenumerate}
  \item \label{Tooa}
The sequence $(T_n)\xoo$ converges
if and only if there exists $\xnu\in \cP(\Moo)$ such
that,
as \ntoo,
\begin{align}
\tauoo(T_n) 
\to \xnu
\qquad\text{in $\cP(\Moo)$},
\end{align}
\ie, if and only if the infinite random matrices
$\rhooo(\xi\nn_1,\xi\nn_2,\dots;T_n)$ converge in distribution, 
where $\xi\nn_i$ are
\iid{} random points in $T_n$ with  $\xi\nn_i\sim\mu_n$.
  \item \label{Toob}
The sequence $(T_n)\xoo$ converges to $D$
if and only if 
as \ntoo,
\begin{align}
\tauoo(T_n) 
\to \tauoo(D)
\qquad\text{in $\cP(\Moo)$},
\end{align}
\ie, if and only if the infinite random matrices
$\rhooo(\xi\nn_1,\xi\nn_2,\dots;T_n)$ converge in distribution
to $\rhooo(\xi_1,\xi_2,\dots;D)$,
where $\xi\nn_i$ are as in \ref{Tooa}
and $\xi_i$ are \iid{} with $\xi_i\sim\nu$.
   
  \end{romenumerate}
In particular, the same results holds for a sequence $(c_nT_n)\xoo$ of
rescaled finite trees.
\end{theorem}

\begin{proof}
  This follows from the remarks before the theorem.
Note that if \eqref{d2} holds for every $r\ge1$, then $(\xnu_r)_r$
is a consistent sequence,
since $(\tau_r(T_n))_r$ is for every $n$.
\end{proof}

\section{Abstract tree limits}\label{Sabs}

Based on the preceding section, we can define tree limits in an
abstract way as follows, using only (part of) the definitions and elementary
considerations above and none
of the deep results of \cite{ET}. (Cf.\ \cite{SJ209} for graph limits.)

Let $\cTf$ be the set of all rescaled finite trees $cT$ (with arbitrary $c>0$).
Then $\tauoo:\cTf\to\cpmoo$.
Let $\fTf:=\tauoo(\cTf)\subseteq\cpmoo$
and
\begin{align}\label{fT}
\fT:=\overline{\fTf}=\overline{\tauoo(\cT)}\subseteq\cpmoo.  
\end{align}
This defines $\fT$ as a closed subset of the Polish space $\cpmoo$; thus
$\fT$ is a Polish space.
Hence, we can regard $\fT$ as a (complete and separable) metric space
whenever convenient; if necessary we can define a metric of $\fT$ \eg{} as
the Prohorov metric on $\cP(\Moo)$ 
\cite[Appendix III]{Billingsley}, \cite[Theorem 8.3.2]{Bogachev},
but we have in the present paper no need for a specific choice of metric.

We can identify a rescaled finite tree $cT$ with its image
$\tauoo(cT)\in\fT$
(temporarily ignoring the question whether this is a one-to-one map).
Then 
convergence as in \refD{D1} is, by \refT{Too}, the same as convergence in
the metric space $\fT$. Furthermore, $\fT$ is the set of all possible limits
of convergent sequences; thus it is natural to say that 
$\fT$ is the set of \emph{tree limits}.

We have thus defined a set of tree limits; moreover, this set has turned out to
be a Polish space.

Similarly, a measured real tree $T$ defines an element $\tauoo(T)\in\cpmoo$.
We define $\cTr$ as the set of all measured real trees and 
$\fTr:=\tauoo(\cTr)\subset\cpmoo$. (We  ignore the set-theoretic difficulty
of defining the "set of all measurable real trees"; formally  we either
consider trees that are subsets of some huge universe, or suitable
equivalence classes under isomorphisms.)
Then the following holds.

\begin{theorem}
  \label{TTT}
With notations as above,
\begin{align}
  \fTf\subseteq \fTr\subseteq\fT=\overline{\fTf}=\overline{\fTr}.
\end{align}
\end{theorem}
We postpone the proof.
It follows that convergence of measured real trees as in
\refD{D2} also is the same as convergence in $\fT$.
From now on, whenever convenient, we identify finite trees and measured real
trees with their images in $\fT$.

Returning to the deep results by \citet{ET} in \refTs{ET1}--\ref{ET3},
we first note that, similarly, each long dendron $D$ defines an element
$\tauoo(D)\in\cpmoo$. \refT{ET3} and the remarks in \refS{Soo} show that
$\tauoo(D)=\tauoo(D')$ if and only if $D$ and $D'$ are isomorphic.
Thus, letting $\fD$ be the set of all
equivalence classes of 
long dendrons modulo isomorphism,
$\tauoo:\fD\to\cpmoo$ is injective.
\begin{theorem}
  \label{TDT}
$\tauoo(\fD)=\fT$, and
the mapping
$\tauoo:\fD\to\fT$ is a bijection.
\end{theorem}
\begin{proof}
If $D$ is a long dendron, then by \refT{ET2}, there exists a convergent
sequence of rescaled finite trees $(c_nT_n)_n$ that converges to $D$.
In other words, $\tauoo(c_nT_n)\to\tauoo(D)$. Thus
$\tauoo(D)\in\overline{\tauoo(\cTf)}=\fT$. 

Conversely, 
  if $\mu\in\fT$, then there exists a sequence $c_nT_n\in\cTf$ such that
  $\tauoo(c_nT_n)\to \mu$. Thus the sequence $c_nT_n$ is convergent, and
by \refT{ET1},   there exists a long dendron $D$ such that $c_nT_n\to D$,
which means $\tauoo(c_nT_n)\to\tauoo(D)$.
Consequently, $\mu=\tauoo(D)$.

Hence, $\tauoo(\fD)=\fT$, and we have already remarked that $\tauoo$ is
injective on $\fD$ by \refT{ET3}.
\end{proof}

Consequently, we can identify $\fD$ and $\fT$, and regard also $\fD$ as the
set of all tree limits. 
(As done by \citet{ET}.)
Note that this defines a topology on $\fD$, making $\fD$ into a Polish space. 

We ignore the taking of equivalence classes,
and regard $\fD$ as the set of all long dendrons. 
Thus, the topology on $\fD$ gives a
notion of convergence for long dendrons. 

\begin{theorem}\label{TD0}
  Let $D=(T,d,\nu)$ and $D_n=(T_n,d_n,\nu_n)$, $n\ge1$ be long dendrons. 
Then the following are  equivalent.
  \begin{romenumerate}
  \item\label{TD0a}  
$D_n\to D$ in $\fD$.
  \item \label{TD0b}
$\tauoo(D_n)\to\tauoo(D)$ in $\cpmoo$.    
  \item\label{TD0c}  
The infinite random matrices
$\rhooo(\xi\nn_1,\xi\nn_2,\dots;D_n)$ converge in distribution
to $\rhooo(\xi_1,\xi_2,\dots;D)$,
where $\xi\nn_i$ are
\iid{} random points in $A_{D_n}$ with  $\xi\nn_i\sim\mu_n$.
and $\xi_i$ are
\iid{} random points in $A_{D}$ with $\xi_i\sim\nu$.

\item\label{TDod} 
$\tau_r(D_n)\to\tau_r(D)$ in $\cP(M_r)$, for every $r\ge1$.

\item \label{TD0e}
 The finite random matrices
$\rho_r(\xi\nn_1,\xi\nn_2,\dots;D_n)$ converge in distribution
to $\rho_r(\xi_1,\xi_2,\dots;D)$
for every $r\ge1$,
where $\xi\nn_i$ and $\xi_i$ are as in \ref{TD0c}.
  \end{romenumerate}
\end{theorem}

\begin{proof}
Immediate by the definitions and comments before the theorem together with
\eqref{erika}.
\end{proof}

Summarizing,
we may thus regard finite trees, real trees, and long dendrons as elements
of the Polish space $\fT\subset\cpmoo$.
This gives a unified meaning to convergence of trees and real trees to a
long dendron, and also a notion of convergence of long dendrons.

We turn to the question whether $\tauoo$ is injective (up to obvious
isomorphisms) on the sets $\cTf$ of finite trees and $\cTr$ of measured real
trees; recall that for long dendrons, this is answered (positively) by
\refT{ET3}. 
\citet[$3\frac12.5$ and $3\frac12.7$]{Gromov} 
studied a more general setting and
proved
that if $X_1=(X_1,d_1,\mu_1)$ and 
$X_2=(X_2,d_2,\mu_2)$ are two separable and complete metric
measure spaces such that the measures have full support,
and $\tauoo(X_1)=\tauoo(X_2)$, then
$X_1$ and $X_2$ are isomorphic.
This applies immediately to rescaled finite trees, and it follows that
if $c_1T_1$ and $c_2T_2$ are rescaled
trees with $\tauoo(c_1T_1)=\tauoo(c_2T_2)$, then $T_1\cong T_2$
as metric spaces, and thus as trees,
and $c_1=c_2$. (Except in the trivial case $|T_1|=|T_2|=1$, when $c_1$ and
$c_2$ are arbitrary.)
In other words, $\tauoo:\cTf\to\fT$ is injective up to isomorphism.

For measured real trees $(T,d,\mu)$, this is not quite true, since it may happen
that $\mu$ is concentrated on a subtree $T'\subset T$, and then
$\tauoo(T,d,\mu)=\tauoo(T',d,\mu)$.
However, if $\cTc$ is the set of measured real trees such that every branch has
positive measure, then $\tauoo$ is injective on $\cTc$ (up to isomorphism).
One way to see that is to note that every $T=(T,d,\mu)\in\cTc$ may be
regarded as a long dendron as in \refE{EOO},
and then use \refT{ET3}.

In general, given a measured real tree $T$, we may prune branches of measure
0 and obtain a subtree $T'\in\cTc$; this is called the \emph{core} of $T$ in
\cite{ET}, where a detailed definition is given.  
We see that the mapping $\tauoo$ does not distinguish between
a measured real tree $T$ and its core $T'$.

In other words, our identification of measured real trees with tree limits
in $\fT$ means that we ignore  branches of measure 0, and thus
identify a tree with its core, but trees with different cores are distinguished.
With some care, we may thus also regard measured real trees as elements of
$\fT$. 

One important consequence of regarding trees, measured real trees and long
dendrons as elements of the Polish space $\fT$ is that then
standard theory (\eg{} \cite{Billingsley})
defines for us random trees, random measured real trees and random long
dendrons, as well as 
convergence in probability or distribution of such random objects.
This will be a central topic in the remainder of the paper.

First, however,
it remains to prove \refT{TTT}.

\begin{proof}[Proof of \refT{TTT}]
First,  as explained in \refE{Ereal},
a rescaled finite tree $cT\in\cTf$ can be embedded in a measured real tree
$c\rT\in\cTr$
such that \eqref{ctt} holds for all finite $r$, and thus also for
$r=\infty$.
This proves $\fTf\subseteq\fTr$.

Recalling \eqref{fT}, it remains only to show $\fTr\subseteq\fT$.

We give first a short proof using the results of \cite{ET}.
If $T$ is a measured real tree, then the constant sequence $T,T,\dots$
trivially is 
convergent, and thus \refT{ET1} shows that there exists a long dendron $D$
such that $T\to D$, which by \refT{TD0} means $\tauoo(T)=\tauoo(D)$.
Hence, by \refT{TDT},
\begin{align}
\tauoo(T)=\tauoo(D)\in\tauoo(\fD)=\fT.  
\end{align}

We give also  an alternative, elementary proof. We do this in several steps.
We consider for simplicity, as said earlier, only separable trees.

\stepx
As in \cite{ET}, say that a measured real tree is a \emph{finite real tree}
if it can be obtained from a finite tree by regarding each edge as an
interval of some positive length (not necessarily the same for all edges),
and adding a probability measure on the (finite) set of vertices.
Note that a finite real tree has finite diameter.
\cite[Lemma 7.2]{ET} shows that every finite real tree $T$ of diameter $\le1$ 
is a limit of rescaled finite trees, \ie, $T\in\fT$.
By rescaling, the same holds for every finite real tree.

\stepx
Suppose that $T=(T,d,\mu)$ is a measured real tree such that $\mu$ is
concentrated on a finite set of points \set{x_1,\dots,x_m}.
Let $T':=\bigcup_{i=1}^m[x_1,x_i]$ be the subtree spanned by \set{x_1,\dots,x_m}.
Then $(T',d,\mu)$ is a finite real tree. Furthermore,
$\tauoo(T)=\tauoo(T',d,\mu)\in\fT$, using Step 1.

\stepx
Suppose that $T=(T,d,\mu)$ is a measured real tree such that $\mu$ is
concentrated on a countable  set of points \set{x_1,x_2,\dots}.
Let 
\begin{align}
  \mu_n:=\sumkn \mu\set{x_k}\gd_{x_k} 
+\Bigpar{\sum_{k=n+1}^\infty \mu\set{x_k}}\gd_{x_1},
\end{align}
where $\gd_x$ is the point mass at $x$.
Let $\xi_i$ be \iid{} with  $\xi_i\sim\mu$,
and let
\begin{align}
  \xi\nn_i:=
  \begin{cases}
    \xi_i& \text{if }\xi_i\in\set{x_1,\dots,x_n},
\\
x_1 &\text{otherwise}.
  \end{cases}
\end{align}
Then $(\xi\nn_i)_i$ are \iid{} random points in $T$ with  $\xi\nn_i\sim\mu_n$.
Furthermore, $\P(\xi\nn_i\neq\xi_i)\to0$ as \ntoo,
and thus $\rho_r(\xi\nn_1,\dots,\xi\nn_r)\pto \rho_r(\xi_1,\dots,\xi_r)$
for each $r\ge1$.
Hence, 
$\tau_r(T,d,\mu_n)\to\tau_r(T,d,\mu)$ for every
finite $r$, and thus also
$\tauoo(T,d,\mu_n)\to\tauoo(T,d,\mu)$. 
Since $\tauoo(T,d,\mu_n)\in\fT$ by Step 2, it follows that
$\tauoo(T,d,\mu)\in\fT$.  

\stepx
Let $T=(T,d,\mu)$ be any separable tree.
There exists a countable dense subset $A:=\set{x_1,x_2,\dots}$.

For each $n\ge1$, define a measurable function $f_n:T\to A$ such that
$d(x,f_n(x))<1/n$ for all $x$. (For example, let $f_n(x):=x_i$ for the smallest 
$i$ such that $d(x,x_i)<1/n$.)
Let $\xi_i$ be \iid{} random points in $T$ with  $\xi_i\sim\mu$,
and let $\xi\nn_i:=f_n(\xi_i)$.
Then $(\xi\nn_i)_i$ are \iid{} with  $\xi\nn_i\sim \mu_n:=f_n(\mu)$, which is
concentrated on the countable set $A$.
By Step 3, $\tauoo(T,d,\mu_n)\in\fT$ for every $n$.
Furthermore, 
\begin{align}
  \bigabs{d(\xi\nn_i,\xi\nn_j)-d(\xi_i,\xi_j)}
\le d(\xi\nn_i,\xi_i)+d(\xi\nn_j,\xi_j)<2/n
\end{align}
for every $i$ and $j$, and 
and thus $\rho_r(\xi\nn_1,\dots,\xi\nn_r)\asto \rho_r(\xi_1,\dots,\xi_r)$ as
\ntoo{} 
for each $r\ge1$.
Hence, 
$\tau_r(T,d,\mu_n)\to\tau_r(T,d,\mu)$ for every
finite $r$, and thus also
$\tauoo(T,d,\mu_n)\to\tauoo(T,d,\mu)$. 
Consequently, $\tauoo(T,d,\mu)\in\fT$. 
\end{proof}

\section{Compactness}\label{Scompact}.

Recall that a set $S$ in a metric space $X$
is \emph{relatively compact} if
every sequence in $S$ has a convergent subsequence. 
(This is equivalent to $\overline S$ being compact.)

Recall also that a family $\set{Z_\ga:\ga\in\cA}$ of random variables in a
metric 
space $X$ is \emph{tight} if for every $\eps>0$ there exists a compact set
$K_\eps\subseteq X$ such that $\P(Z_\ga\notin K_\eps)< \eps$ for every
$\ga\in\cA$. In this case we also say that the family of distributions
$\set{\cL(Z_\ga)}$ is tight.

Prohorov's theorem \cite[Section 6]{Billingsley} says that for a Polish
space $X$, 
the set of distributions $\set{\cL(Z_\ga)}$ is relatively compact in
$\cP(X)$ if and only if $\set{Z_\ga}$ is tight.
In particular, this leads to the following characterization of relative
compactness in $\fT$.

\begin{theorem}\label{TC}
  Let $A=\set{c_\ga T_\ga:\ga\in\cA}$ be a set of rescaled trees.
Then the following are equivalent, where 
$(\xi\nnga_i)_i$ are 
\iid{} 
uniformly random vertices in $T_\ga$
and $d_\ga$ is the graph distance in $T_\ga$.
\begin{romenumerate}
  
\item \label{TCa}
$A$ is relatively compact.
\item \label{TCtau}
The set of measures 
$\set{\tauoo(c_\ga T_\ga):\ga\in\cA}\subseteq\cpmoo$ is tight. 
\item \label{TCrho}
The set of random variables
$\set{\rhooo(\xi\nnga_1,\xi\nnga_2\dots; c_\ga T_\ga):\ga\in\cA}$ in $\Moo$
is tight. 
\item \label{TC2}
The set of random variables $\set{c_\ga d_\ga(\xi\nnga_1,\xi\nnga_2):\ga\in\cA}$ 
is tight.
\item \label{TC1}
There exists $x_\ga\in T_\ga$, $\ga\in\cA$, such that
the set of random variables $\set{c_\ga d_\ga(\xi\nnga_1,x):\ga\in\cA}$ is tight.
\end{romenumerate}
The same holds, mutatis mutandis, for sets of measured real trees 
\set{(T_\ga,d_\ga,\mu_\ga)}, and for
sets of long dendrons \set{(T_\ga,d_\ga,\nu_\ga)};
in these cases, $\xi\nnga_i\sim\mu_\ga$ and $\xi\nnga_i\sim\nu_\ga$,
respectively, 
and for long dendrons we use $d_D$ defined in \eqref{dD}.
\end{theorem}

\begin{proof}
\ref{TCa}$\iff$\ref{TCtau}$\iff$\ref{TCrho}: 
Prohorov's theorem, 
together with  the definition of convergence and \refT{Too}.

\ref{TCrho}$\implies$\ref{TC2}: 
Immediate by \eqref{rhor}, since the mapping $(a_{ij})_{i,j}\mapsto a_{1,2}$
is continuous $\Moo\to\bbR$.

\ref{TC2}$\implies$\ref{TCrho}: 
Follows by symmetry and the fact that $\Moo$
has the product topology.
To be more precise, let $\eps>0$. By \ref{TC2}, there exist constants $C_k$,
$k\ge0$,
such that $\P\bigpar{c_\ga d_\ga(\xi\nnga_1,\xi\nnga_2)>C_k} < 2^{-k}\eps$
for every $\ga$. Then $K:=\set{(a_{ij})_{i,j}:|a_{ij}| \le C_{i+j}}$ is a
compact subset of $\Moo$, and, by symmetry,
\begin{align}\label{moore} 
  \P\bigpar{\rhooo(\xi\nnga_1,\xi\nnga_2\dots; c_\ga T_\ga)\notin K}
&
\le \sum_{i,j=1}^\infty\P\bigpar{c_\ga d_\ga(\xi\nnga_i,\xi\nnga_j)>C_{i+j} }
\notag\\&
< \sum_{i,j=1}^\infty2^{-i-j}\eps=\eps.
\end{align}

\ref{TC2}$\implies$\ref{TC1}: 
If $C_\eps$ is such that
$\P\bigpar{c_\ga d_\ga(\xi\nnga_1,\xi\nnga_2)>C_\eps} <\eps$, then
(by Fubini's theorem), there exists $x_\ga$ such that
$\P\bigpar{c_\ga d_\ga(\xi\nnga_1,x_\ga)>C_\eps} <\eps$.

\ref{TC1}$\implies$\ref{TC2}: 
If $C_\eps$ is such that
$\P\bigpar{c_\ga d_\ga(\xi\nnga_1,x_\ga)>C_\eps} <\eps/2$,
then
$\P\bigpar{c_\ga d_\ga(\xi\nnga_1,\xi\nnga_2)>2C_\eps} <\eps$.
\end{proof}

\begin{definition}\label{Dtight}
  A set 
of rescaled trees, measured real trees, or long dendrons, is \emph{tight} if
\ref{TC2} (or, equivalently, \ref{TC1}) in \refT{TC} holds.
\end{definition}

With this definition, \refT{TC} simply says that a set of rescaled trees,
measured 
real trees or long dendrons is relatively compact if and only if it is tight.
Usually, we consider sequences rather than general sets, and then \refT{TC}
has the following corollary.
\begin{corollary}
  \label{CC}
If $(c_nT_n)_n$ is a tight sequence of rescaled trees, then
some subsequence converges to some long dendron.

The same holds for tight sequences of measured real trees and for tight
sequences of long dendrons.
\end{corollary}

\section{Simple examples}\label{Sex}

As a preparation for the study of limits of random trees in
the following sections,
we give here a few simple examples of limits of deterministic trees.

\begin{example}[paths]\label{EPn}
  Let $P_n$ be the path with $n$ vertices. We may take the vertices to be
  \set{1,\dots,n}, and then $d_{P_n}(x,y)=|x-y|$.

Let $I=(I,d,\mu)$ be the unit interval $I:=\oi$ considered as a measured
real tree with the usual metric $d$ and Lebesgue measure $\mu$.
We regard $I$ as a long dendron as in \refE{ED=T},
and  claim that $\frac{1}{n} P_n\to I$.

To see this, let $(\xi_i)_i$ be \iid{} with $\xi_i\sim\mu$,
and let $\xi\nn_i:=\ceil{n\xi_i}$.
Then $(\xi\nn_i)_i$ are \iid{} uniform vertices of $P_n$, and
$\frac{1}{n}\xi\nn_i\to\xi_i$ as \ntoo.
Hence, \eqref{rhor} shows that 
\begin{align}
  \rho_r\bigpar{\xi\nn_1,\dots,\xi\nn_r;\tfrac{1}{n}P_n}
\asto \rho_r\bigpar{\xi_1,\dots,\xi_r;I}
\end{align}
for every $r$. This implies convergence in distribution, and thus
$\tau_r\bigpar{\tfrac1nP_n}\to\tau_r(I)$, and thus
\begin{align}\label{Pn}
  \frac{1}{n}P_n\to I.
\end{align}

The diameter of $P_n$ is $n-1$. Obviously, we obtain the same limit $I$
if we use the \ETn{} $\frac{1}{n-1}P_n$.
(Note that the limit $I$ is a short dendron.)
\end{example}

\begin{example}[stars]\label{ESn}
Let $S_n=K_{n-1,1}$ be a star with $n$ vertices.
If $(\xi\nn_i)_i$ are \iid{} random vertices in $S_n$, then with probability
$1-O(1/n)$, $\xi\nn_1$ and $\xi\nn_2$ are distinct peripheral vertices, and
thus $d(\xi\nn_1,\xi\nn_2)=2$. Hence,
$d(\xi\nn_1,\xi\nn_2)\pto 2$ as \ntoo, and thus \refT{TOO} shows that
\begin{align}\label{Sn}
  S_n\to \OO_1.
\end{align}

Of course, we can use the \ETn{} and consider
$\frac12S_n$, which has diameter 1, and obtain the equivalent result
$\frac12S_n\to\OO_{1/2}$. 
\end{example}

\begin{example}[complete binary trees]\label{EBn}
Let $B_n$ be a complete binary tree with height $n-1$ and thus $2^n-1$ vertices.
Let $o$ be the root and let
$h(x):=d(x,o)$ (known as the depth of $x$) 
denote the distance from a vertex $x$ to the root.

If $(\xi\nn_i)_i$ are \iid{} random vertices in $B_n$, then for $0\le k<n$,
\begin{align}\label{bn1}
  \P\bigpar{h(\xi\nn_i)<n-k} = \frac{2^{n-k}-1}{2^n-1}\le 2^{-k}.
\end{align}
Since also $h(\xi\nn_i)<n$,
it follows that 
\begin{align}\label{bn2}
\frac{1}n h(\xi\nn_i)\pto1.  
\end{align}
Recall that $x\bmin y$ denotes the last common ancestor of $x,y\in B_n$.
If $h(x\bmin y)\ge k$, then $x$ and $y$ are both descendants of one of the
$2^k$ vertices $z$ with depth $k$. 
For each $z$, the number of such $x$ (or $y$) is $2^{n-k}-1$. Hence,
\begin{align}\label{bn3}
  \P\bigpar{h(\xi\nn_1\bmin\xi\nn_2)\ge k}
=2^k\frac{(2^{n-k}-1)^2}{(2^n-1)^2}
\le 2^{-k}, \qquad k<n.
\end{align}
Consequently,
\begin{align}\label{bn4}
\frac{1}{n} h(\xi\nn_1\bmin\xi\nn_2)\pto0.
\end{align}
Since $d(x,y)=h(x)+h(y)-2h(x\bmin y)$ for $x,y\in B_n$, 
\eqref{bn2} and \eqref{bn4} imply
\begin{align}
\frac{1}{n} d(\xi\nn_1,\xi\nn_2)
=\frac{1}{n} h(\xi\nn_1)
+ \frac{1}{n} h(\xi\nn_2)
-\frac{2}{n} h(\xi\nn_1\bmin\xi\nn_2)
\pto 1+1-0=2.
\end{align}
Consequently, \refT{TOO} yields
\begin{align}\label{Bn}
  \frac{1}{n}B_n\to \OO_1.
\end{align}
We see that \eqref{Bn} encapsulates (and formalizes) the fact that almost
all pairs of vertices in $B_n$ have distance $\approx 2n$.

Recall that $B_n$ has $N:=2^n-1$ vertices. Thus, \eqref{Bn} can also be written
\begin{align}\label{Bn2}
  \frac{1}{\log N} B_n\to\OO_{1/\log 2}.
\end{align}

The results extend to complete $b$-ary trees $T^b_n$, for any $b\ge2$, with
$N=(b^{n}-1)/(b-1)$ nodes. In this case,
\begin{align}\label{Bnb}
  \frac{1}{\log N} T^b_n\to\OO_{1/\log b}.
\end{align}
\end{example}

\begin{example}[superstars]\label{Esuper}
Let $T_n$ consist of a central vertex $o$ with $n$ paths attached: 
$N_{kn}$ paths with $k$ edges for $k\ge1$, 
all having $o$ as one endpoint but otherwise
disjoint, for some numbers $N_{kn}\ge0$ with $\sum_k N_{kn}=n$.
The number of vertices is thus $|T_n|=1+\sum_k k N_{kn}$.
We assume that as \ntoo, for some $p_k\ge0$ with $\sumk p_k=1$,
\begin{align}\label{ss0}
\frac{  N_{kn}}n \to p_k, \qquad k\ge1,
\end{align}
and 
\begin{align}\label{ss1}
  \sum_k  k\frac{N_{kn}}n \to \gam:=\sumk kp_k <\infty.
\end{align}
Thus 
\begin{align}\label{ssE}
|T_n|\sim\gam n.  
\end{align}
Suppose further (this actually follows from the other assumptions)
that
\begin{align}\label{ss2}
  \sum_k  k^2\frac{N_{kn}}n=o(n).
\end{align}
Let $(\xi\nn_i)_i$ be \iid{} uniformly random vertices of $T_n$.
It follows from the assumptions above that, for  $k\ge1$,
\begin{align}\label{ssa}
  \P\bigpar{d(\xi\nn_i,o)=k} =\frac{ \sum_{j\ge k} N_{jn}}{|T_n|}
\to\frac{\sum_{j\ge k} p_j}{\gam} =:q_k.
\end{align}
Note that 
\begin{align}\label{ssq}
  \sumk q_k = \frac{\sumk \sum_{j\ge k} p_j}{\gam} 
=\frac{\sumj jp_j}{\gam} =1.
\end{align}
Let $\nu$ be the probability distribution on $\bbN$ given by $\nu\set
k=q_k$,
and let $(\xi_i)_i$ be \iid{} with $\xi_i\sim\nu$.
Then, \eqref{ssa} shows that 
\begin{align}\label{tess}
d(\xi\nn_i,o)\dto\xi_i,
\qquad  \ntoo.  
\end{align}
Furthermore, for any $i,j\ge1$, by \eqref{ss2} and \eqref{ssE},
as in the special case in \refE{ESn},
\begin{align}\label{tekk}
&  \P\bigpar{d(\xi\nn_i,\xi\nn_j)\neq d(\xi\nn_i,o)+d(\xi\nn_i,o)}
\notag\\&\quad
= \P\bigpar{\text{$\xi\nn_i$ and $\xi\nn_j$ are in the same path}}
= \frac{\sum_k k^2 N_{kn} }{|T_n|^2}\to 0.
\end{align}
It follows from \eqref{tess} and \eqref{tekk} that for any $r\ge1$,
\begin{align}\label{tokk}
  \rho_r\bigpar{\xi\nn_1,\dots\xi\nn_r;T_n}
\dto \bigpar{(\xi_i+\xi_j)\indic{i\neq j}}_{i,j=1}^r
=\rho\bigpar{\xi_1,\dots,\xi_r;\OO_\nu}.
\end{align}
Hence,
\begin{align}\label{Tn}
  T_n \to \OO_\nu.
\end{align}
\end{example}

\section{Limits of random trees}\label{SRT}

In the rest of the paper we consider limits of random (finite) trees.
Suppose that $\ctn$, $n\ge1$, are random trees (with any distributions)
and let, conditioned on
$\ctn$, $(\xi\nn_i)_i$ be \iid{} uniformly random vertices of $\ctn$.
We are concerned with limits in distribution or probability of $c_n\ctn$
to some random or deterministic long dendron (tree limit). (Here, $c_n$ are
some given positive numbers.)
By the definitions above, this is equivalent to convergence of the
conditional distributions
\begin{align}\label{boa}
  \tauoo(c_n\ctn)=\cL\bigpar{\rhooo(\xi\nn_1,\xi\nn_2,\dots;c_n\ctn)\mid\ctn},
\end{align}
regarded as random elements of $\cP(\Moo)$;
we thus want to show
either that $\tauoo(c_n\ctn)$ converges in distribution to $\tauoo(D)$ for a
random long dendron $D$, or (as a special case) that it converges in
probability to 
$\tauoo(D)$ for a fixed $D$.

\begin{remark}
  It is important that we consider randomness in two steps: first $\ctn$ is a random
  tree and then $(\xi\nn_i)_i$ are random vertices in $\ctn$. As seen in
  \eqref{boa}, we are interested in the
\emph{quenched} version, where we first sample $\ctn$ and then condition on
$\ctn$.

The alternative \emph{annealed} version considers $\ctn$ and $(\xi\nn_i)_i$
as random together; the annealed distribution of 
$\rhooo(\xi\nn_1,\xi\nn_2,\dots;c_n\ctn)$ is the mean 
(or intensity) $\E\tauoo(c_n\ctn)$
of the random measure in \eqref{boa},
which in general is not what we want.
\end{remark}

We note one simple case where the difference between quenched and annealed
disappears. 

\begin{theorem}\label{TOOX}
Let  $(\ctn)_n$ be a sequence of rescaled random trees and 
$c_n$ some positive numbers.
Let further $a\ge0$.
Then the following are equivalent,
where $d_n$ is the graph distance in $\ctn$ and
$(\xi\nn_i)_i$ are \iid{}
uniformly random vertices in $\ctn$.
\begin{romenumerate}
  
\item \label{TOOXa}
$c_nT_n\pto \OO_a$.
\item \label{TOOXb}
\begin{align}\label{tooxb}
  c_nd_n(\xi\nn_1,\xi\nn_2)\pto 2a.
\end{align}

\item \label{TOOXc}
For every $\eps>0$,
\begin{align}\label{tooxc}
\P\bigsqpar{|c_nd_n(\xi\nn_1,\xi\nn_2)- 2a|>\eps\mid\ctn} \pto0.
\end{align}
\end{romenumerate}
\end{theorem}

\begin{proof}
Recall that for any random variables $Z_n$,
\begin{align}\label{pto0}
  Z_n\pto0 \iff \E \bigsqpar{|Z_n|\bmin1}\to0.
\end{align}
Thus,
  for deterministic trees $T_n$, the convergence in probability \eqref{too}
is equivalent to
\begin{align}\label{boc}
\E\bigsqpar{| c_nd_n(\xi\nn_1,\xi\nn_2)- 2a|\land1}\to0.
\end{align}
Consequently, by \refT{TOO}, \ref{TOOXa} is equivalent to
\begin{align}\label{bod}
\E\bigsqpar{| c_nd_n(\xi\nn_1,\xi\nn_2)- 2a|\land1\mid\ctn}\pto0.
\end{align}
A simple argument using Markov's inequality shows that \eqref{bod} is
equivalent to \ref{TOOXc}. 
(This argument is a conditional version of \eqref{pto0}.)

Furthermore,
since the \lhs{} of \eqref{bod} is bounded by 1,
\eqref{pto0} shows that \eqref{bod} is equivalent to
\begin{align}\label{boe}
\E\bigsqpar{| c_nd_n(\xi\nn_1,\xi\nn_2)- 2a|\land1}\pto0,
\end{align}
which by a final application of \eqref{pto0} is equivalent to \ref{TOOXb}.
\end{proof}

We give also a version of the compactness criterion in \refT{TC} for random
trees.
We  state the theorem for a sequence of random trees,
although the statement and proof holds for an arbitrary set.

\begin{theorem}\label{TCX}
  Let $\xpar{c_n \ctn}_n$ be a sequence of rescaled random trees.
Then the following are equivalent, where 
$(\xi\nn_i)_i$ are 
\iid{} 
uniformly random vertices in $\ctn$
and $d_n$ is the graph distance in $\ctn$.
\begin{romenumerate}
  
\item \label{TCXa}
The sequence $(c_n\ctn)_n$
of random elements of\/ $\fT$ is relatively compact in $\cP(\fT)$.
\item \label{TCXb}
The sequence $(c_n\ctn)_n$
of random elements of\/ $\fT$ is tight.
\item \label{TCX2}
The sequence of random variables $\bigpar{c_n d_n(\xi\nn_1,\xi\nn_2)}_n$ 
is tight.
\end{romenumerate}
\end{theorem}

\begin{proof}
\ref{TCXa}$\iff$\ref{TCXb}: 
  Since $\fT$ is a Polish space, 
this is Prohorov's   theorem.

\ref{TCXb}$\implies$\ref{TCX2}:
By the definitions in \refS{Sabs}, $\fT$ is a closed subspace of $\cP(\Moo)$
and it follows that \ref{TCXb} means that for every $\eps>0$, there
exists a compact set $\cK_\eps\subset\cP(\Moo)$ such that, for every $n\ge1$,
\begin{align}\label{daa}
\P\bigpar{  \tauoo(c_n\ctn)\notin \cK_\eps}<\eps.
\end{align}
Furthermore, Prohorov's theorem (now applied to the Polish space $\Moo$)
shows that for every $\gd>0$, there exists a
compact set $K_{\eps,\gd}\subset\Moo$ such if $\gl\in\cP(\Moo)$,
then
\begin{align}\label{dab}
\gl\in\cK_\eps \implies
  \gl(K_{\eps,\gd})>1-\gd.
\end{align}
Since the projection $(a_{ij})_{ij}\mapsto a_{12}$ is continuous $\Moo\to\bbR$,
there exists a constant $C_{\eps,\gd}$ such that if 
$(a_{ij})_{ij}\in K_{\eps,\gd}$, then $|a_{12}|\le C_{\eps,\gd}$.

Consequently, for every $\eps,\gd>0$, and all $n$,
\begin{align}\label{dac}
{|c_nd_n(\xi\nn_1,\xi\nn_2)|>C_{\eps,\gd}}
\implies
{\rhooo(\xi\nn_1,\dots;c_n\ctn)\notin K_{\eps,\gd}}
\end{align}
and thus, using also \eqref{dab},
\begin{align}\label{dad}
&{\P\bigpar{|c_nd_n(\xi\nn_1,\xi\nn_2)|>C_{\eps,\gd}\mid\ctn}\ge\gd}
\notag\\&\quad
\implies
{\P\bigpar{\rhooo(\xi\nn_1,\dots;c_n\ctn)\notin K_{\eps,\gd}\mid\ctn}\ge\gd}
\notag\\&\quad
\implies
\tauoo(c_n\ctn)=
\cL\bigpar{\rhooo(\xi\nn_1,\dots;c_n\ctn)} \notin \cK_{\eps}.
\end{align}
Hence, \eqref{daa} implies
\begin{align}\label{dae}
\P\Bigpar{\P\bigpar{|c_nd_n(\xi\nn_1,\xi\nn_2)|>C_{\eps,\gd}\mid\ctn}\ge\gd}<\eps
\end{align}
which yields
\begin{align}\label{daf}
\P\bigpar{|c_nd_n(\xi\nn_1,\xi\nn_2)|>C_{\eps,\gd}}
=\E\P\bigpar{|c_nd_n(\xi\nn_1,\xi\nn_2)|>C_{\eps,\gd}\mid\ctn}
\le\gd+\eps.
\end{align}
By taking $\gd=\eps$,
this shows \ref{TCX2}.

\ref{TCX2}$\implies$\ref{TCXb}:
By \ref{TCX2}, for every $\eps>0$, there exists $C_{\eps}$ such that
\begin{align}\label{dag}
\P\bigpar{|c_nd_n(\xi\nn_1,\xi\nn_2)|>C_{\eps}}
<\eps.
\end{align}
Define 
\begin{align}\label{dak}
  K_{\eps}:=\bigset{(a_{ij})_{ij}\in\Moo:|a_{ij}|\le C_{2^{-i-j}\eps}}.
\end{align}
This is a compact subset of $\Moo$, and \eqref{dag} implies
\begin{align}\label{dah}
\P\bigpar{\rhooo(\xi\nn_1,\dots;c_n\ctn)\notin K_{\eps}}
&\le\sum_{i,j=1}^\infty
\P\bigpar{|c_nd_n(\xi\nn_i,\xi\nn_j)|>C_{2^{-i-j}\eps}}
\notag\\&
<\sum_{i,j=1}^\infty2^{-i-j}\eps=\eps.
\end{align}
Hence,
\begin{align}\label{dal}
\E\bigsqpar{\P\bigpar{\rhooo(\xi\nn_1,\dots;c_n\ctn)\notin K_{\eps}\mid\ctn}}
=&{\P\bigpar{|c_nd_n(\xi\nn_1,\xi\nn_2)|\notin K_{\eps}}}<\eps,
\end{align}
and Markov's inequality shows that, for any $\ell\ge1$,
\begin{align}\label{dam}
\P\bigsqpar{
\P\bigpar{\rhooo(\xi\nn_1,\dots;c_n\ctn)\notin K_{4^{-\ell}\eps}\mid\ctn}>2^{-\ell}}
<2^{-\ell}\eps.
\end{align}
By the definition of $\tauoo$, this is the same as
\begin{align}\label{dan}
\P\bigsqpar{\tauoo(c_n\ctn)\bigpar{\Moo\setminus  K_{4^{-\ell}\eps}}>2^{-\ell}}
<2^{-\ell}\eps.
\end{align}

Let
\begin{align}
  \cK_\eps:=\bigset{\gl\in\cP(\Moo):\gl(K_{4^{-\ell}\eps})\ge 1-2^{-\ell}, 
\forall\ell\ge1}
\end{align}
and note that $\cK_\eps$ is compact by Prohorov's theorem.
It follows by \eqref{dan} that
\begin{align}
\P\bigpar{\tauoo(c_n\ctn)\notin \cK_\eps}
  <\suml 2^{-\ell}\eps =\eps.
\end{align}
Hence, the sequence $\tauoo(c_n\ctn)$ is tight in $\cP(\Moo)$, and thus in $\fT$.
\end{proof}

\begin{remark}
Again, 
the same holds, mutatis mutandis, for random measured real trees 
and for random  long dendrons.
In fact, the argument is quite general and holds for any measured metric spaces.
We believe that this may be known, but we do not know a reference and have
included a full proof for completeness.
\end{remark}

\section{Conditioned Galton--Watson trees, I}\label{SGW1}

Consider a \GWp{} with some given offspring distribution $\zeta$.
(We let $\zeta$ denote both the distribution and a random variable with this
distribution.) 
The family tree of the \GWp{} is a random tree $\cT$, which in the
subcritical and critical cases (\ie, when $\E\zeta\le1$) is \as{} finite.
$\cT$ is a \emph{\GWt}, and the random tree $\cT_n:=(\cT\mid |\cT|=n)$
obtained by conditioning $\cT$ on a given size $n$ is said to be a
\emph{\cGWt}. (We consider only $n$ such that $\P(|\cT|=n)>0$.)
For further details, see \eg{} the survey \cite{SJ264}.

In the standard case $\E\zeta=1$ and $\Var\zeta<\infty$,
\citet{AldousI,AldousII,AldousIII} proved convergence 
in distribution
of the \cGWt{}
$\cT_n$, after rescaling, to a limit object called
the \emph{\BCRT}; this is a random measured real tree which we
denote by $\cTex$, for reasons given below. 
Aldous's original result was not in terms of the type of convergence discussed
in the present paper, but it holds in the present context too.
In fact, Aldous's result has been stated in several different forms, more
or less equivalent; one 
version, stated \eg{} in 
\cite[Theorem 8]{HaasMiermont} and \cite[Theorem 5.2]{AddarioEtAl},
is convergence in the
\GHP{} metric
(defined in \eg{} \cite[Chapter 27]{Villani} and \cite[Section 6]{Miermont}), 
which is stronger than \GP{}
convergence 
and thus implies convergence in the tree limit sense used in the present
paper
(see \refR{RGP} and \refE{ED=T}).
We thus have the following.

\begin{theorem}\label{TGW1a}
  Let $\cT_n$ be a \cGWt{} with critical offspring distribution $\zeta$ with
finite variance, \ie, we assume $\E\zeta=1$ and $\gss:=\Var\zeta\in(0,\infty)$.
Then, as \ntoo, 
\begin{align}\label{tgw1a}
  \frac{1}{\sqrt n} \cT_n \dto \frac{1}{\gs} \cTex,
\end{align}
where $\cTex$ is the \BCRT.
\end{theorem}

\begin{proof}
  As said before the theorem, this is known.
For completeness, we sketch a proof in the present context; 
omitted details can be found \eg{} in
\eg{} \cite{AldousIII} and \cite{LeGall2005}.

One standard version of Aldous's theorem uses the \emph{contour function}
$C_{\cT_n}(t)$ of $\cT_n$. In general, if $T$ is a rooted tree with $|T|=n$,
then $C_T$ is a continuous function $[0,2(n-1)]\to\ooo$;
informally, $C_T(t)$ is the distance, at time $t$, from the root to a
particle that travels with unit speed along the ``outside'' of the tree,
starting at the root at time $0$ and returning at time $2(n-1)$, having
traversed every edge once in each direction.
\citet{AldousIII} showed that
\begin{align}\label{ald}
  \frac{1}{\sqrt n} C_{\cT_n}(2(n-1)t) \dto \frac{2}{\gs}\Bex(t)
\qquad\text{in }C\oi,
\end{align}
where $\Bex(t)$ is the standard Brownian excursion, which is a random
continuous function $\oi\to\ooo$ with $\Bex(0)=\Bex(1)=0$.

Every continuous function $g:\oi\to\ooo$ with $g(0)=g(1)=0$ 
defines a real tree $\Tx_g$: define a pseudometric on $\oi$ by 
\begin{align}\label{dst}
  d(s,t):=g(s)+g(t)-2\min_{u\in[s,t]}g(u),
\qquad 0\le s\le t\le 1,
\end{align}
and form the quotient of $\oi$ 
modulo the equivalence relation $\set{d(s,t)=0}$;
see \eg{} \cite[Theorem 2.2]{LeGall2005}. 
The uniform (Lebesgue) measure on $\oi$ induces a
measure $\mu$ on $\Tx_g$, making $(T_g,\mu)$ a measured real tree.

Taking $g(t)=C_{T}(2(n-1)t)$ for a rooted tree $T$ with $|T|=n$ 
gives $T_g=\rT$, the real tree obtained from $T$ as in \refE{Ereal}.
The measure $\mu$ induced by $g$ is the uniform measure on the edges of
$\rT$, and not the uniform measure $\mu'$ on the vertices of $T\subseteq\rT$;
however, it is easy to couple these measures and find $\xi\sim\mu$ and
$\xi'\sim\mu'$ such that $\P(|\xi-\xi'|>1)\le 1/n$.
It follows that \eqref{ald} implies \eqref{tgw1a}, both in the sense of
the present paper and in the stronger \GHP{} metric.
\end{proof}

\begin{remark}
  \citet{Duquesne2003} 
considered the case when 
$\zeta$ has infinite variance and furthermore
is in the domain of attraction of a stable
distribution; he extended Aldous's result
and showed convergence of the contour process of $c_n\ctn$ 
(for suitable $c_n$)
to a certain stochastic process in this case too; this implies
convergence of $c_n\ctn$ to a random real tree called the \emph{stable tree}
\cite{LeGall2006} in \GHP{} sense,
and thus in the weaker sense of tree
limits,
also in this case.
(See  \cite[Theorem 8]{HaasMiermont}, 
with a somewhat stronger assumption on $\zeta$.)
\end{remark}

\begin{remark}\label{RGWtEx}
  As is well-known, 
several important classes of random trees can be represented as \cGWt{s}
$\ctn$ satisfying the conditions above
by choosing suitable offspring distributions $\zeta$;
thus \refT{TGW1a} applies to them.
This includes (uniformly) random
labelled trees ($\gss=1$), random ordered trees $\gss=2$) and random binary
trees ($\gss=1/2$); see \eg{} \cite{AldousII} and \cite{SJ264}.
\end{remark}

\begin{remark}\label{Rsgt}
Recall that random \emph{simply generated trees}
are defined by a weight sequence $(w_k)_k$;
see, again, \eg{} \cite{AldousII} or \cite{SJ264}
for the definition and for the well-known fact
that
while simply generated trees   are  more general than \cGWt{s}, 
they can  in many cases
be reduced to
equivalent \cGWt{s}.
Thus \refT{TGW1a} applies to simply generated trees under rather weak
conditions.
In other cases of simply generated trees, the results in \refSs{SGW2} and
\ref{SIII} may apply.
\end{remark}

\section{Conditioned Galton--Watson trees, II}\label{SGW2}

Although a large class of \cGWt{s} (and simply generated trees) are covered
by \refT{TGW1a}, there are also other cases.
One class of \cGWt{s} with a different local limit behaviour 
showing condensation was found by
\citet{JonssonS}; this was generalized  in \cite{SJ264}, with further
results in \cite{Stufler}.
This class of \cGWt{s} (called type II in \cite{SJ264} and \cite{Stufler})
has offspring distributions $\zeta$ satisfying
\begin{align}
 & 0<\kk:=\E\zeta<1,\label{gwIIa}
\\
&\E R^\zeta=\infty, \qquad R>1.\label{gwIIb}
\end{align}
In other words, the \GWt{s} are subcritical, and $\zeta$ has infinite \mgf;
see further \cite[Section 8]{SJ264}.
We will show that this class  has  tree limits 
that are very different from the ones in \refS{SGW1}.

For a rooted tree $T$ and a vertex $v\in T$, let $\ddd(v)$ denote the
outdegree of $v$. Furthermore, let 
$\DDD=\DDD(T):=\max_{v\in T}\ddd(v)$ be the maximum outdegree, and let
$\vvv$ be the vertex with maximum outdegree
(chosen as \eg{} the lexicographically first if there is a tie), so
$\DDD=\ddd(\vvv)$.

It is shown in \cite[Section 19.6]{SJ264} that \eqref{gwIIa}--\eqref{gwIIb}
 imply the
existence (asymptotically) of one or several vertices of very high
(out)degree,
with a total outdegree $\approx(1-\Ezeta)n$;
typically, there is one single large vertex with degree $\approx(1-\Ezeta)n$,
but this is not always the case; see \cite{SJ264} and \refR{R2large}.
We will assume that there is such a vertex; a case known as 
\emph{complete  condensation}. 
To be precise, we assume that $\zeta$ is such that
\begin{align}\label{gwIIc}
  \DDD(\ctn)=(1-\Ezeta)n+\op(n).
\end{align}
For example, this hold when
the offspring distribution 
satisfies 
\eqref{gwIIa}--\eqref{gwIIb}
and has a power law tail, as shown by
\citet[Theorem 5.5]{JonssonS}, see also
\cite[Theorem 19.34]{SJ264} and
(more generally, with regularly varying tails)
\citet[Theorem~1]{Kortchemski}.

We note  that \eqref{gwIIc} implies that the second largest outdegree is
$\op(n)$; in particular the maximum degree vertex $\vvv$ is unique \whp,
see \cite[paragraph after Lemma 19.32]{SJ264}.

\begin{theorem}\label{TGW2}
Let $\ctn$ be a  \cGWt{} with subcritical offspring distribution $\zeta$
satisfying \eqref{gwIIa}--\eqref{gwIIb}
and \eqref{gwIIc}.
Then, as \ntoo,
\begin{align}\label{tgw2}
  \ctn\pto\OO_\nu,
\end{align}
where $\nu=\Ge(1-\Ezeta)$ is a geometric distribution
on $\bbN:=\set{1,2,\dots}$.
\end{theorem}

\begin{remark}
  Note that there is no rescaling of $\ctn$ in \eqref{tgw2};
the situation is similar to \refEs{ESn} and \ref{Esuper}.
Distances are typically small; formally, the distance $d(\xi\nn_1,\xi\nn_2)$
between two random vertices is stochastically bounded (\ie, tight).
Hence, the local limits studied in \cite{SJ264} and \cite{Stufler} are
essentially global in this case.
\end{remark}

\begin{remark}
  The diameter  $\diam(\ctn)\pto\infty$, \eg{} by \refL{LIIc} below.
Hence, rescaling such that the diameter becomes 1 would only give the trivial
limit $\OO_0$, see \refR{ROO0}.
\end{remark}

The rest of this section contains the proof of \refT{TGW2}.
  We begin with some further notation.
In this proof, all trees are rooted and ordered.
Trees that are equal up to order-preserving isomorphisms are regarded as equal.
Let $\TT$ be the countable set of all finite trees.

Let again $\cT$ denote the (unconditioned) \GWt{} with the chosen
offspring distribution $\zeta$. Since $\E\zeta<1$, $\cT$ is \as{} finite.
If $\ttt$ is any fixed finite tree, let
\begin{align}
  \pi_\ttt:=\P(\cT=\ttt).
\end{align}
In other words, $(\pi_\ttt)_{\ttt\in\TT}$ is the probability distribution of 
$\cT\in\TT$.

The \emph{fringe tree} \cite{Aldous-fringe}
of a tree $T$ at a vertex $v$, denoted $T^v$, 
is the subtree of $T$ consisting of $v$ and its descendants, 
rooted at $v$.

Let $\ttt$ denote a finite tree.
For any tree $T$, let 
\begin{align}
  N_{\ttt}(T):=\bigabs{\set{v\in T: T^v= \ttt}},
\end{align}
\ie,  the number of fringe trees of $T$
equal to $\ttt$.

It is shown in \cite[Theorem 7.12]{SJ264} that for any fixed tree $\ttt$,
assuming \eqref{gwIIa}--\eqref{gwIIb},
\begin{align}\label{nypa}
  \frac{N_\ttt(\ctn)}{n}\pto \pi_\ttt.
\end{align}
In other words, the conditional distribution of $\ctn^v$ given $\ctn$,
with $v$ a random vertex, converges in probability to the distribution of $\cT$.
 Both sides of \eqref{nypa} 
are probability distributions on the countable set of
finite trees, and we claim that  it follows 
that 
the random distribution $\bigpar{N_\ttt(\ctn)/{n}}_\ttt$ converges in
probability to $(\pi_\ttt)_\ttt$ in total variation, and thus
for any set $\TTx\subseteq\TT$ of finite trees,
\begin{align}\label{nyttt}
\sum_{\ttt\in\TTx}  \frac{N_\ttt(\ctn)}{n}\pto \sum_{\ttt\in\TTx}\pi_\ttt
=\P(\cT\in\TTx).
\end{align}
To see this, note
the corresponding result for sequences of probability distributions on a
countable set is well known, see \eg{} \cite[Theorem 5.6.4]{Gut}.
The
version used here with random (conditional) distributions and  convergence
in probability follows by essentially the same proof, or by first using the 
Skorohod coupling theorem \cite[Theorem~4.30]{Kallenberg}, to see that we
may assume that \eqref{nypa} holds \as, and then using the deterministic
version. 

We need an extension of \eqref{nypa}. Let
\begin{align}\label{nypak}
  N_{\ttt,k}(T):=\bigabs{\set{v\in T: T^v= \ttt \text{ and } \ddd(\hv)=k}},
\end{align}
where $\hv$ denotes the parent of $v$ (undefined for the root).
Also, let
\begin{align}
  p_k:=\P(\zeta=k), \qquad k\ge0
\end{align}
and note that
\begin{align}\label{kk}
  \Ezeta:=\E\zeta=\sumk kp_k.
\end{align}

\begin{lemma}
  \label{LIIa}
Assume \eqref{gwIIa}--\eqref{gwIIb}.
For every fixed $\ttt\in\TT$ and $k\in\bbN$,
\begin{align}\label{nypon}
  \frac{N_{\ttt,k}(\ctn)}{n}\pto k p_k \pi_\ttt.
\end{align}
\end{lemma}
\begin{proof}
Let, for $j=1,\dots,k$,
\begin{align}\label{nypj}
  N_{\ttt,k,j}(T):=
\bigabs{\set{v\in T: T^v= \ttt, \ddd(\hv)=k,\text{ and $v$ is the $j$th
  child of $\hv$} }},
\end{align}
Note that $v$ is in the set in \eqref{nypj} if and only if
$T^{\hv}\in\TT_j$, where $\TT_j$ is the set of all trees $\httt$ such that
the root has exactly $k$ children, and if $w$ is the $j$th of these, then
the fringe tree  $\httt^w=\ttt$. 
Hence, \eqref{nyttt} shows that, 
using also the recursive property of the \GWt{} $\cT$,
\begin{align}\label{nyttto}
 \frac{N_{\ttt,k,j}(\ctn)}{n}
\pto \P(\cT\in\TT_j)
= p_k\P(\cT=\ttt)=p_k\pi_\ttt.
\end{align}
The result follows, since 
$N_{\ttt,k}(\ctn)=\sum_{j=1}^kN_{\ttt,k,j}(\ctn)$.
\end{proof}

We have so far not used the assumption \eqref{gwIIc}, but it is essential
for the next lemma. 
Recall that $\gD=\gD(\ctn)$ is the maximum outdegree,
and that \whp{} $\vvv$ is the only vertex of outdegree $\DDD$.
Hence, \whp, $N_\DDD=1$ and
$N_{\ttt,\DDD}(\ctn)$ is the number of children $v$ of $\vvv$
such that $\ctn^v=\ttt$.
\begin{lemma}
  \label{LIIb}
Assume \eqref{gwIIa}--\eqref{gwIIc}.
For every fixed $\ttt\in\TT$,
\begin{align}\label{liib}
  \frac{N_{\ttt,\DDD}(\ctn)}{n}\pto (1-\Ezeta) \pi_\ttt.
\end{align}
\end{lemma}

\begin{proof}
  Let $N_k:=\bigabs{\set{v\in\ctn:\ddd(v)=k}}$. Then, as a consequence of
  \eqref{nyttt} or as a simpler version of \eqref{nypa}, see
\cite[Theorem 7.11]{SJ264},
\begin{align}\label{kya}
  N_k/n\pto p_k,
\qquad k\ge0.
\end{align}

Let $\eps>0$, and choose $K$ such that 
\begin{align}
  \label{keps}
\sum_{k>K}kp_k<\eps.
\end{align}
The number of vertices having a parent of outdegree $k$ is $kN_k$.
Thus $\sum_k kN_k=n-1$. Hence,
using \eqref{kya}, \eqref{gwIIc}, \eqref{kk}  and \eqref{keps},
and assuming as we may that $N_\DDD=1$,
\begin{align}\label{nyq}
  \sum_{K+1}^{\DDD-1} kN_k&
=\sum_{k=1}^\infty kN_k -  \sum_{k=1}^K kN_k -\DDD
\notag\\&
=n-1- \sum_{k=1}^K kp_kn - \xpar{1-\Ezeta} n + \op(n)
\notag\\&
= \sum_{k>K} kp_kn + \op(n)
\notag\\&
<\eps n+\op(n)
<2\eps n
\qquad\text{\whp}
\end{align}

Now consider the $N_\ttt$ vertices $v$ such that $\ctn^v=\ttt$.
Assume for convenience $n>|\ttt|$, so that the root is not one of these
vertices. Then, using \eqref{nypa} and \eqref{nypon},
\begin{align}\label{nyr}
  N_{\ttt,\DDD}&
=N_\ttt-\sum_{k=1}^K N_{\ttt,k} - \sum_{K+1}^{\DDD-1} N_{\ttt,k}
\notag\\&
=
\pi_\ttt n -\sum_{k=1}^K kp_k\pi_\ttt n - \sum_{K+1}^{\DDD-1} N_{\ttt,k} + \op(n).
\end{align}
Thus, using \eqref{nyq} and $N_{\ttt,k}\le k N_k$, \whp,
\begin{align}
  \pi_\ttt n -\sum_{k=1}^K kp_k\pi_\ttt n - 2\eps n + \op(n)
\le N_{\ttt,\DDD}
\le \pi_\ttt n -\sum_{k=1}^K kp_k\pi_\ttt n  + \op(n).
\end{align}
Using also \eqref{keps} and \eqref{kk}, we find that \whp
\begin{align}
(1-\Ezeta)\pi_\ttt n -3\eps n
\le{N_{\ttt,\DDD}}
\le (1-\Ezeta)\pi_\ttt n +2\eps n.
\end{align}
The result \eqref{liib} follows, since $\eps>0$ is arbitrary.
\end{proof}

Next, for a tree $\ttt$, and $\ell\ge0$, let $w_\ell$ be the number of
vertices at distance $\ell$ from the root.
Furthermore, for $\ell\ge1$,
let $W_\ell:=w_\ell(\ctn^{\vvv})$, the number of vertices in $\ctn$ that are
descendants of $\vvv$ and are $\ell$ generations from it,
and let $\xW:=n-\sum_{\ell\ge1} W_\ell$ be the number of vertices that are not
descendants of $\vvv$.

\begin{lemma}
  \label{LIIc}
Assume \eqref{gwIIa}--\eqref{gwIIc}.
Then, 
for $\ell\ge1$,
\begin{align}\label{liic}
  \frac{W_\ell}{n} \pto \xpar{1-\Ezeta}\Ezeta^{\ell-1}
\end{align}
and
\begin{align}\label{liicx}
  \frac{\xW}{n} \pto0.
\end{align}

\end{lemma}

\begin{proof}
  We have, assuming $N_\DDD=1$ which holds \whp,
  \begin{align}\label{laa}
W_\ell = \sum_{\ttt\in\TT} N_{\ttt,\DDD}(\ctn)w_{\ell-1}(\ttt).
  \end{align}
For any finite family $\TT_0\subset\TT$,
by \refL{LIIb},
\begin{align}\label{lab}
\frac{1}{n}  \sum_{\ttt\in\TT_0} N_{\ttt,\DDD}(\ctn)w_{\ell-1}(\ttt)
\pto
\xpar{1-\Ezeta} \sum_{\ttt\in\TT_0} \pi_{\ttt}w_{\ell-1}(\ttt).
\end{align}
 Hence, by \eqref{laa},
\begin{align}\label{lac}
\frac{1}{n} W_\ell
=\frac{1}{n}  \sum_{\ttt\in\TT} N_{\ttt,\DDD}(\ctn)w_{\ell-1}(\ttt)
\ge
\xpar{1-\Ezeta} \sum_{\ttt\in\TT_0} \pi_{\ttt}w_{\ell-1}(\ttt)
+\op(1)
\end{align}
for any finite $\TT_0$.
Furthermore, by elementary branching process theory,
\begin{align}\label{laf}
   \sum_{\ttt\in\TT} \pi_{\ttt}w_{\ell-1}(\ttt)
=\E w_{\ell-1}(\cT)
=(\E\zeta)^{\ell-1}
=\Ezeta^{\ell-1}.
\end{align}
In particular, the sum converges, and it follows from \eqref{lac} that
\begin{align}\label{lad}
\frac{1}{n} W_\ell
\ge
\xpar{1-\Ezeta} \sum_{\ttt\in\TT} \pi_{\ttt}w_{\ell-1}(\ttt)
+\op(1).
\end{align}
Thus, \eqref{lad} yields
\begin{align}\label{lag}
\frac{1}{n} W_\ell
\ge
\xpar{1-\Ezeta}\Ezeta^{\ell-1}+\op(1),
\qquad \ell\ge1.
\end{align}
We can sum \eqref{lag} over any set of $\ell$, using the same 
argument as for \eqref{lad} again. In particular, we obtain
\begin{align}\label{laj}
\frac{1}{n}\sum_{j\neq \ell} W_j
\ge
\xpar{1-\Ezeta}\sum_{j\neq\ell}\Ezeta^{j-1}+\op(1)
.\end{align}
On the other hand, trivially,
\begin{align}\label{lai}
\frac{1}{n}\sumj W_j
\le 1=
\xpar{1-\Ezeta}\sumj\Ezeta^{j-1}.
\end{align}
Subtracting \eqref{laj} from \eqref{lai} yields
\begin{align}\label{lak}
\frac{1}{n} W_\ell
\le
\xpar{1-\Ezeta}\Ezeta^{\ell-1}+\op(1)
,\end{align}
which together with \eqref{lag} yields the result \eqref{liic}.

Furthermore, \eqref{lag} and \eqref{laj} yield
\begin{align}\label{lal}
\frac{1}{n}\sumj W_j
\ge
\xpar{1-\Ezeta}\sumj \Ezeta^{j-1}+\op(1)
=1+\op(1)
.\end{align}
Thus,
\begin{align}\label{lam}
\xW=
  n-\sumj W_j =\op(n),
\end{align}
which yields \eqref{liicx} and completes the proof.
\end{proof}

\begin{lemma}\label{LIId}
  Assume \eqref{gwIIa}--\eqref{gwIIc}.
Let $(\xi\nn_i)_i$ be \iid{} vertices in $\ctn$.
Then
\begin{align}\label{liid}
  \P\Bigpar{d(\xi\nn_1,\xi\nn_2)\neq
d(\xi\nn_1,\vvv)+d(\xi\nn_2,\vvv)\Bigm|\ctn}
\pto0.
\end{align}
\end{lemma}

\begin{proof}
If $\xi\nn_1$ and $\xi\nn_2$ both are descendants of $\vvv$, then   
$d(\xi\nn_1,\xi\nn_2)= d(\xi\nn_1,\vvv)+d(\xi\nn_2,\vvv)$
unless $\xi\nn_1$ and $\xi\nn_2$ are in the same fringe subtree rooted
at a child of $\vvv$. Hence, the probability in \eqref{liid} is at most
\begin{align}\label{lb}
  2\frac{\xW}{n}+\frac{1}{n^2}\sum_{\ttt\in\TT} N_{\ttt,\DDD}|\ttt|^2.
\end{align}
By \eqref{liicx}, it suffices to show that the sum in \eqref{lb} is
$\op(n^2)$.
Fix $K>1$, and let $\TTK:=\set{\ttt\in\TT:|\ttt|\le K}$
and
$\TTKK:=\set{\ttt\in\TT:|\ttt|> K}$.
First, deterministically,
\begin{align}\label{lba}
  \sum_{\ttt\in\TTK} N_{\ttt,\DDD}|\ttt|^2
\le K   \sum_{\ttt\in\TTK} N_{\ttt,\DDD}|\ttt|
\le Kn = o(n^2).
\end{align}
Secondly, since no subtree of $\ctn$ has more than $n$ vertices,
\begin{align}\label{lbb}
  \sum_{\ttt\in\TTKK} N_{\ttt,\DDD}|\ttt|^2
\le n  \sum_{\ttt\in\TTKK} N_{\ttt,\DDD}|\ttt|.
\end{align}
By \eqref{liib}, since the set $\TTK$ is finite,
\begin{align}\label{lbc}
    \sum_{\ttt\in\TTKK} N_{\ttt,\DDD}|\ttt|
\le  n -   \sum_{\ttt\in\TTK} N_{\ttt,\DDD}|\ttt|
= n - n   \sum_{\ttt\in\TTK} (1-\Ezeta)\pi_\ttt|\ttt|+\op(n).
\end{align}
On the other hand
\begin{align}\label{lbd}
  \sum_{\ttt\in\TT} (1-\Ezeta)\pi_\ttt|\ttt| = (1-\Ezeta)\E|\cT|
= (1-\Ezeta)\frac{1}{1-\Ezeta}=1.
\end{align}
Thus, for every $\eps>0$, we may choose $K$ such that
$ \sum_{\ttt\in\TTK} (1-\Ezeta)\pi_\ttt|\ttt|>1-\eps$,
and then \eqref{lbb}--\eqref{lbc} yield, \whp,
\begin{align}\label{top}
  \sum_{\ttt\in\TTKK} N_{\ttt,\DDD}|\ttt|^2 \le n^2\eps+\op(n^2) < 2\eps n^2.  
\end{align}
The result follows by \eqref{lb}, \eqref{lba} and \eqref{top}.
\end{proof}

\begin{proof}[Proof of \refT{TGW2}]
Let $(\xi\nn_i)$ be \iid{} uniformly random vertices of $\ctn$, and
let $Y\nn_i:=d(\xi\nn_i,\vvv)$.
Then \refL{LIIc} yields
\begin{align}\label{lca}
  \P\bigpar{Y\nn_i=\ell\mid\ctn}\pto (1-\Ezeta)\Ezeta^{\ell-1}, 
\qquad \ell\ge1,
\end{align}
and \refL{LIId} yields
\begin{align}\label{lcb}
  \P\bigpar{d(\xi\nn_i,\xi\nn_j)\neq Y\nn_i+Y\nn_j\mid\ctn}\pto0.
\end{align}
We may for convenience, by the
Skorohod coupling theorem \cite[Theorem~4.30]{Kallenberg}, 
or (more elementary) by considering suitable subsequences,
assume that \eqref{lca} and \eqref{lcb} hold with $\asto$.
Then, \eqref{lca} 
and the independence of $(Y\nn_i)_i$ shows that
\as{} the sequence $\ctn$ is
such that,
conditioned on $\ctn$, we have
$(Y\nn_i)_i\dto (Y_i)_i$ with $Y_i\sim\Ge (1-\Ezeta)$
i.i.d.
 Consequently, for every $r\ge1$, using also \eqref{lcb},
\begin{align}
  \rho_r\xpar{\xi\nn_1,\dots,\xi\nn_r;\ctn}
\dto \bigpar{(Y_i+Y_j)\indic{i\neq j}}_{i,j=1}^r
= \rho_r(\xi_1,\dots,\xi_r;\OO_\nu),
\end{align}
where $\xi_i:=(\bullet,Y_i)\in A_{\OO_\nu}$ are \iid{} with $\xi_i\sim\nu$. 
Hence, \as, $\tau_r(\ctn)\to \tau_r(\OO_\nu)$ and thus $\ctn\to\OO_\nu$.
\end{proof}

\begin{remark}\label{R2large}
  \cite[Example 19.37]{SJ264} gives an example of an offspring distribution
  satisfying \eqref{gwIIa}--\eqref{gwIIb} but not
\eqref{gwIIc}. In this example, there exists a subsequence of $n$ such that 
$\gD(\ctn)/n\pto0$; there exists also another subsequences for which $\ctn$
\whp{} contains two vertices of outdegree $n/3$.

It is an open problem to find tree limits in such cases.
In the case just mentioned with two large vertices (but not more), we
conjecture that the tree limit is similar to $\OO_\nu$ in \refE{EOO}, 
but has a base
consisting of a unit interval with the marginal distribution of $\nu$
concentrated on the two endpoints.
\end{remark}

\begin{remark}
  The proof above is based on the result \eqref{nypa} for random fringe
  trees $\ctn^v$ of $\ctn$. However, we also consider the parent of the
  random node $v$, see \eg{} \eqref{nypak}; thus we really consider
  properties of (part of) the \emph{extended fringe tree}, also defined by
  \citet{Aldous-fringe}. The asymptotic distribution of the entire extended
  fringe tree was found by \citet{Stufler}. However, his result is for the
  annealed version, where the tree $\ctn$ and the vertex $v$ are chosen
  at random together, while we here need the quenched version, where we
fix (\ie,   condition on) $\ctn$ and then take a random vertex $v$.
We have therefore used the (quenched) result \eqref{nypa} rather than the
result of \cite{Stufler}.
In fact, the argument above is easily extended to show that for 
the part of the extended
fringe tree up to the first very large ancestor (\ie, \whp, $\vvv$),
the infinite limit tree found by \citet{Stufler} is also the limit in the
quenched sense. However, this does not hold for the remaining part of the
extended fringe tree; this part is, for $n-\op(n)$ choices of $v$, equal to 
the part of $\ctn$ between the root and $\vvv$, and conditioned on $\ctn$
it is thus \whp{} equal to some random tree determined by $\ctn$.
Consequently, there is no quenched limit of the entire extended fringe tree.
\end{remark}

\section{Simply generated trees, type III}\label{SIII}
As said in \refR{Rsgt}, many simply generated trees are covered by the
results for \cGWt{s} in the preceding sections. However, there are also
simply generated trees of a different type (called type III in
\cite{SJ264}), where there is no equivalent \cGWt.
These are defined by  weight sequences $(w_k)_k$ such that the power series
$\sum_k w_k z^k$ 
has radius of convergence 0, \ie,
\begin{align}
  \label{III}
\sumko w_k z^k=\infty,\qquad z>0.
\end{align}
As shown in \cite{SJ264}, such simply generated trees have many
similarities with \cGWt{s} satisfying \eqref{gwIIa}--\eqref{gwIIb}, 
if we define $\kk:=0$.
In particular, there exists one or several vertices of high outdegree, with
total outdegree $n-\op(n)$. Again, \cf{} \eqref{gwIIc},
we regard as typical the case of complete
condensation now defined by
\begin{align}\label{IIIc}
  \DDD(\ctn)=n-\op(n),
\end{align}
so that there is a single vertex $\vvv$ that has fathered almost all others.
(In fact, then \whp{} $\vvv$ is the root, see \cite[(20.2)]{SJ264}.)
We then have an almost trivial result.
\begin{theorem}\label{TIII}
  Let $\ctn$ be a simply generated tree defined by a weight sequence
  $(w_k)\xoo$ satisfying \eqref{III} and \eqref{IIIc}.
Then
\begin{align}\label{tiii}
  \ctn\pto \OO_1.
\end{align}
\end{theorem}
\begin{proof}
If $(\xi\nn_i)$ are \iid{} uniformly random vertices of $\ctn$, then 
\eqref{IIIc} shows that \whp{} $\xi\nn_1$ and $\xi\nn_2$ are children of the
node $\vvv$ with highest degree. 
Furthermore, \whp, $\xi\nn_1\neq\xi\nn_2$.
Hence, \whp{} $d(\xi\nn_1,\xi\nn_2)=2$, and the result follows by \refT{TOOX}.
\end{proof}

The proof shows that \eqref{tiii} holds because $\ctn$ ``almost'' is a star
$S_n$, see \refE{ESn}.

\begin{example}
  Considered the case $w_k=k!$, which satisfies \eqref{III}.
It was shown in \cite{SJ259} that then \eqref{IIIc} holds.
More precisely, 
if the fringe trees rooted at children of the root are called \emph{branches},
\whp{} 
$\ctn$ has a root of degree $n-1-Z_n$, where $Z_n\dto\Po(1)$, 
and of the $n-1-Z_n$ branches,
$Z_n$ have size 2 and all others are single vertices (\ie, leaves of
$\ctn$).

 The case $w_k=(k!)^\ga$ with $0<\ga<1$ is similar \cite{SJ259};
there are more branches
that have size $\ge2$, and the largest may have size $\ceil{1/\ga}+1$,
but their number is still $\op(n)$ and \eqref{IIIc} holds.
If  $w_k=(k!)^\ga$ with $\ga>1$, then $\ctn=S_n$ \whp.

Thus, \eqref{tiii} holds for any $\ga>0$.
\end{example}

\begin{remark}
\cite[Examples 19.18 and 19.39]{SJ264} give examples where \eqref{IIIc}
does not hold, and there are (at least for some subsequences)  several large
vertices. It is still true that \whp{} almost all vertices are at distance 1
from one of the large vertices, so possible subsequence limits (in
distribution) of $\ctn$ are determined by the structure of the subtree
spanned by the large vertices. We leave further study of this case as an
open problem.
\end{remark}

\section{Logarithmic trees}\label{Slog}

Many random trees $\ctn$ have heights that \whp{} are of order $\log n$;
we call such trees \emph{logarithmic trees}.
(Here, as usual, $n$ measures the size of the tree in some sense.
Note, however, that in some examples below, 
 $|\ctn|$ is random and not always
equal to $n$; nevertheless, it is always \whp{} of order $n$.)
Some examples are binary search trees, random recursive trees,
$m$-ary search trees, digital search trees, preferential attachment
trees and tries. 
Two general classes of such trees (overlapping, and together
including the examples just mentioned) are studied in \refSs{Ssplit} 
and~\ref{SCMJ}. 

In all these cases, it turns out that
the random trees $\ctn$ 
after rescaling 
have a non-random tree limit (in probability)
of the type $\OO_a$ in \refE{EOO}.
More precisely, for some $a\in(0,\infty)$,
\begin{align}\label{tlog}
  \frac{1}{\log n}\ctn \pto \OO_a.
\end{align}
We note that by \refT{TOOX}, \eqref{tlog} is equivalent to
\begin{align}\label{dlog}
  \frac{\dx(\xi\nn_1,\xi\nn_2)}{\log n}\pto 2a,
\end{align}
where, as usual, $\dx$ is the distance in $\ctn$ and $(\xi\nn_i)_i$ are
\iid{} vertices in $\ctn$. 
Equivalently, from our
point of view, \ie{} with regard to distances between random points, 
these classes of logarithmic trees behave
just like the deterministic binary tree in \refE{EBn}.
(We do not know any natural example of logarithmic random trees
that do not satisfy \eqref{tlog}--\eqref{dlog}.)

\begin{remark}
Note that \refT{TOOX} shows that in this case, with convergence to a
constant,
the annealed result \eqref{dlog} is sufficient. 
To prove \eqref{tlog} for some random trees $\ctn$,
we therefore may work with
annealed results and do not have to show quenched versions (which often
are more difficult).
\end{remark}

Before considering particular classes of random trees, we note the following
simple result, which is used to prove \eqref{tlog} in many cases.

\begin{theorem}\label{Tlog}
  Let $\ctn$ be random trees such that, as \ntoo,
for some $a\in\ooo$,
\begin{align}\label{olof}
  \frac{\dx(\xi\nn_1,o)}{\log n}&\pto a
\end{align}
and
\begin{align}
  \frac{\dx(\xi\nn_1\bmin\xi\nn_2,o)}{\log n}&\pto0\label{twin}
,\end{align}
where $(\xi\nn_i)_i$ are \iid{} random vertices in $\ctn$.
Then \eqref{tlog} and \eqref{dlog} hold.
\end{theorem}

\begin{proof}
We have  
$d(v,w)=d(v,o)+d(w,o)-2d(v\bmin w,o)$ for any $v,w\in\ctn$;
thus \eqref{dlog} follows from \eqref{olof} and \eqref{twin}.
(Cf.\ \refE{EBn}, where the same argument was used.)
\end{proof}

The estimate \eqref{twin} is usually easy, for example by arguments sch as
in \refL{LSB} below, so the main task is to prove \eqref{olof}; this has
been done for many logarithmic random trees.

 The distance $d(\xi\nn_1,\xi\nn_2)$ between two random vertices
has previously been studied in a number of papers for various random trees.
These results verify \eqref{dlog} and thus \eqref{tlog} for several 
random trees; we give some examples.
(The references below show stronger results, which we ignore here.)

\begin{example}\label{EBST}
\emph{Binary search trees} were studied by
\citet{MahmoudN} 
who showed (in particular)
\eqref{dlog} with $a=2$. 
(See also \citet{PP2004}.)
Hence,
\begin{align}
\frac{1}{\log n}\ctn \pto \OO_2.  
\end{align}
This also follows by any of \refT{Tsplit}, \ref{Tsplit2} or \refT{TCMJ} 
below.
\end{example}

\begin{example}\label{ERRT}
\emph{Random recursive trees} were studied by \citet{Pan2004},
who showed (in particular)
\eqref{dlog} with $a=1$. 
Hence,
\begin{align}
\frac{1}{\log n}\ctn \pto \OO_1.  
\end{align}
This also follows by \refT{TCMJ} below.
\end{example}

\begin{example}\label{EPA}
\emph{Heap ordered trees}
(also called 
\emph{plane-oriented recursive trees} and
\emph{preferential attachment trees})
were studied by \citet{MorrisPP}
who showed (in particular) 
\eqref{dlog} with $a=1/2$. 
Hence,
\begin{align}\label{epa}
\frac{1}{\log n}\ctn \pto \OO_{1/2}.  
\end{align}
This also follows by \refT{TCMJ}  below.
\end{example}

\begin{example}\label{Ebinc}
\emph{Random $b$-ary recursive trees} 
(\emph{$b$-ary increasing trees})
were studied by 
\citet{MunsoniusR} who showed (in particular) \eqref{dlog} with $a=b/(b-1)$.
Hence,
\begin{align}
\frac{1}{\log n}\ctn \pto \OO_{b/(b-1)}.  
\end{align}
This also follows by \refT{TCMJ}  below,
or (using \cite[Theorem 6.1]{SJ320}) by \refTs{Tsplit} and \ref{Tsplit2}.
(The binary search tree in \refE{EBST} is the case $b=2$.)
\end{example}

\begin{example}\label{EPAX}
  More generally, for a preferential attachment tree where, in each round, 
a node with outdegree $k$
gets a child with probability proportional to $\chi k+\rho$, 
it follows from \refT{TCMJ} and \cite[Example 6.4]{SJ306} that
\begin{align}\label{epax}
\frac{1}{\log n}\ctn \pto \OO_{\rho/(\chi+\rho)}.    
\end{align}
\refEs{EBST}--\ref{Ebinc} are the cases with $(\chi,\rho)=$
$(-1,2)$, $(0,1)$, $(1,1)$, $(-1,b)$, respectively.
\end{example}

We end with an example that, as far as we know,   
does not follow from the general results in \refSs{Ssplit} and \ref{SCMJ}.
\begin{example}
\emph{Simple families of increasing trees}
(\emph{simply generated increasing trees}) were studied
by \citet{PP2004b}, who showed that if the generating function is a
polynomial of degree $d\ge2$, then \eqref{dlog} holds with $a=d/(d-1)$.
Hence,
\begin{align}
\frac{1}{\log n}\ctn \pto \OO_{d/(d-1)}.  
\end{align}
\end{example}

\section{Split trees}\label{Ssplit}
 Random \emph{split trees} were introduced by \citet{Devroye} as a unified
model that includes many important families of random trees (of logarithmic
height), for example binary search trees, $m$-ary search trees, tries and
digital search trees.
\refT{Tsplit} below  shows that random split trees 
after rescaling
have a non-random tree limit of the type $\OO_a$ in \refE{EOO}.
Equivalently, 
by \refT{TOOX}, distances between random points satisfy \eqref{dlog}.

The definition of split trees involves several parameters $b$, $s$, $s_0$, $s_1$
and a \emph{split vector} $\cV=(V_1,\dots,V_b)$ which is a random vector with
$V_i\ge0$ and $\sumib V_i=1$, \ie, a random probability distribution on
\set{1,\dots,b}.
The split tree is defined as a subtree of the infinite $b$-ary tree  $\Tb$. 
The tree is constructed by  adding a sequence of $n$ \emph{balls} to the
tree, which initially is empty. Each ball  arrives at the root and then
moves recursively as follows; see \cite{Devroye} for further details.

Each vertex is equipped with its own copy $\cV\nnx v$ of the random split vector
$\cV$; these copies are independent.
Each vertex has maximum capacity $s\ge1$;
the first $s$ balls that arrives at a vertex stay there (temporarily),
but when the $(s+1)$th ball arrives at the vertex, 
some balls are sent to its children, leaving $s_0\in[0,s]$ balls
that remain in the vertex for ever. 
(The details of this step depend on $s_1$, see \cite{Devroye}.)  
Any further ball that comes to the vertex is immediately passed along to 
one of its children, with probability $V\nnx v_i$ for child $i$
and independently of all previous events.

The split tree $\ctn$ is defined as the set of all vertices that have been
visited by a ball; note that (if $s_0=0$) some vertices in $\ctn$ may be empty, 
but there is always at least one ball in some descendant of the vertex.

We exclude the trivial case $\max(V_1,\dots,V_b)=1$ \as, and then $\ctn$ is
finite \as.
(Usually one assumes the slightly stronger $V_i<1$ \as{} for every $i$
\cite{Devroye}.) 

It is important to note that $\ctn$ is defined with a fixed number $n$ of
balls, while the number of vertices $|\ctn|$ in general is random.
Nevertheless, it is easy to see that
\begin{align}\label{sp1}
  \E |\ctn| = O(n).
\end{align}
Furthermore, since each node stores at most $s$ balls, we have a
deterministic lower bound
\begin{align}\label{sp2}
  |\ctn|\ge n/s.
\end{align}
In fact, in most cases $\E|\ctn|/n$ converges to some constant, and,
moreover, $|\ctn|/n$ converges in probability to the same constant, see
\cite[Theorem 1.1]{Holmgren2012}.
However, this is not always the case; for some tries, $\E|\ctn|/n$ oscillates.

We define
\begin{align}\label{chi}
  \chi:=\sum_{i=1}^b\E\bigsqpar{V_i\log(1/V_i)},
\end{align}
and note that $0<\chi<\infty$.

\begin{theorem}\label{Tsplit}
  Let $\ctn$ be a random split tree with a split vector
  $\cV=(V_1,\dots,V_b)$
and let $\chi$ be given by \eqref{chi}.
Then,
\begin{align}\label{tsplit}
\frac{1}{\log n}  \ctn\pto \OO_{1/\chi}.
\end{align}
\end{theorem}

\begin{proof}
As said in \refS{Slog},  by \refT{TOOX}, \eqref{tsplit} is equivalent to
  \begin{align}\label{ton}
    \frac{\dx(\xi\nn_1,\xi\nn_2)}{\log n} \pto \frac{2}{\chi}.
  \end{align}
where $(\xi\nn_i)_i$ are \iid{} vertices in $\ctn$. 
Under a technical condition, \eqref{ton} was proved by 
\citet[Corollary 1]{arxiv1902}, as a corollary to some stronger estimates.
(Actually, their $\dx$ is slightly different, and includes the distance to
the root, but the same proof yields \eqref{ton}.)

For completeness, we give a proof (by similar methods)
in the following subsection, not requiring any further conditions; in fact,
we consider there an even more general model.
\end{proof}

Without going into details,
we note that the proof of \eqref{ton}
in \cite{arxiv1902}, as well as our similar proof in \refSS{SSsplit2},
is based on showing the two results
\eqref{olof} and \eqref{twin},
and that \eqref{olof} was shown by \citet{Holmgren2012}.

Note also the related fact that 
if $\eta\nn$ is a random \emph{ball} in $\ctn$, then
\begin{align}\label{luc}
\frac{d(\eta\nn,o)}{\log n}\pto\frac{1}{\chi}.
\end{align}
This was proved by \citet[Theorem 2]{Devroye},
see also the stronger results by
\citet[Theorem 1.3]{Holmgren2012} 
(under a weak technical assumption)  
and \citet[Lemma 13(ii)]{arxiv1902}
(Actually, \citet{Devroye} considered the depth of the last added ball, 
and not a random one, but
that easily implies the result \eqref{luc} by 
\citet[Proposition 1.1]{Holmgren2012}, arguing as in 
\citet[proof of Corollary 1.1]{Holmgren2012}.)

Furthermore, \citet{Ryvkina} showed (in particular)
the corresponding fact for the distance between two random balls:
\begin{align}\label{teta}
    \frac{\dx(\eta\nn_1,\eta\nn_2)}{\log n} \pto \frac{2}{\chi}.
\end{align}
See also \citet[Lemma 3.4]{AlbertH++}, showing that
$\dx(\eta\nn_1\bmin\eta\nn_2,o)$ is tight, which together with \eqref{luc}
implies \eqref{teta}, 

\begin{remark}\label{Rballs}
  In analogy with \refT{TOOX}, we can interpret \eqref{teta} as convergence
\begin{align}\label{balllog}
  \frac{1}{\log n}\bigpar{\ctn,\mu^*_n} \pto \OO_{1/\chi},
\end{align}
where we equip $\ctn$ with the probability measure 
$\mu^*_n$ 
defined as the distribution of the balls on $\ctn$.
\end{remark}

\subsection{Generalized split trees}\label{SSsplit2}

We define random \emph{generalized split trees} as follows;
this is a minor variation of the model in \citet{Broutin+++}.
%
Let $2\le b<\infty$ 
be a fixed  branching factor 
and suppose that for every integer $n\ge1$ we have a 
random vector $\cN\nn=(N\nn_i)_{i=0}^b$ with $N\nn_i\in\bbNo$
and
\begin{align}\label{nna}
  \sumiob N\nn_i=n.
\end{align}
Consider the infinite $b$-ary tree $\Tb$. 
For a given number $n$ of balls, all starting at the root, distribute the
balls according to $\cN\nn$, with $N\nn_0$ balls remaining in the root 
(for ever),
and $N\nn_i$ balls passed to the $i$th child.
Continue recursively in each subtree that has received at least one ball, 
using an independent copy of $\cN\nnx m$ at each vertex that has received $m$
balls. 

It is convenient to begin by equipping each vertex $v$ in the infinite
tree $\Tb$ with a private copy $\cN\nnx{n,v}$ of $\cN\nn$
for each $n\ge1$, with all these random vectors $\cN\nnx{n,v}$ independent.
Then, at each vertex $v$ that receives $m\ge1$ balls, we apply
$\cN\nnx{m,v}$. 

The tree $\ctn$ is defined as the set of all vertices that have received at
least one ball (whether or not any ball remains there).
Equivalently, $\ctn$ is the set of all
vertices $v\in\Tb$ such that the fringe tree $\Tb^v$ contains at least one ball.
Note again that the size $|\ctn|$ is random.

We assume the following:
\begin{PXenumerate}{ST}
  
\item \label{ST1}
There exists a constant $C_0$ such that for every $n$, \as,
  \begin{align}\label{st1}
0\le N\nn_0\le C_0.
  \end{align}

\item\label{ST2} 
The random vector $n\qw\cN\nn$ converges in distribution as \ntoo:
  \begin{align}\label{st2}
\frac{1}{n}\cN\nn = \Bigpar{\frac{N\nn_i}n}_{i=0}^b 
\dto \cV=\bigpar{V_i}_{i=0}^b.
  \end{align}

\item \label{ST3}
For every $n\ge1$,
\begin{align}\label{st3a}
  \P \bigpar{\max_{1\le i\le b} N\nn_i = n} < 1
\end{align}
and, similarly,
\begin{align}\label{st3b}
  \P \bigpar{\max_{1\le i\le b} V_i = 1} < 1.
\end{align}

\end{PXenumerate}
We call the limit $\cV$ in \eqref{st2} the (asymptotic) \emph{split vector}.
Note that $V_0=0$ by \ref{ST1} and \ref{ST2}; thus it suffices to consider
$(V_i)_1^b$. Furthermore, \eqref{st2} implies
\begin{align}\label{sumv}
  \sumib V_i = 1 \qquad \text{\as}
\end{align}
Thus, $\cV$ is a random probability distribution on \set{1,\dots,b}.

It should be clear that the definition above includes the split trees
defined by \citet{Devroye} and discussed above.
(In particular, our model includes tries, unlike the version in
\cite{Broutin+++}.) 

\begin{remark}\label{RST3}
  \eqref{st3a} only excludes the trivial case when, for some $n$, 
\as{} all $n$ balls
  are passed to the same child, and therefore, by induction, 
continue along some infinite path so that $\ctn$ becomes infinite.

Conversely, it is easy to see by induction that \eqref{st3a} implies that
$\ctn$ is finite \as{} for every $n\ge1$. 

Moreover, \eqref{st3b} implies  uniformity in \eqref{st3a}: it is easy to
see that \eqref{st3a}--\eqref{st3b} is equivalent to
the existence of $c,\gd>0$ such that, for every $n\ge1$,
\begin{align}\label{st3c}
  \P \bigpar{\max_{1\le i\le b} N\nn_i >(1-\gd)n} \le 1-c.
\end{align}
\end{remark}

\begin{theorem}\label{Tsplit2}
  Let $\ctn$ be a random generalized split tree with a split vector
  $\cV=(V_1,\dots,V_b)$
and let $\chi$ be given by \eqref{chi}.
Then, 
  \begin{align}\label{tonx}
    \frac{\dx(\xi\nn_1,\xi\nn_2)}{\log n} \pto \frac{2}{\chi},
  \end{align}
where $\dx$ is the distance in $\ctn$ and $(\xi\nn_i)_i$ are
\iid{} vertices in $\ctn$. Consequently,
\begin{align}\label{tsplit2}
\frac{1}{\log n}  \ctn\pto \OO_{1/\chi}.
\end{align}
\end{theorem}

To prove \refT{Tsplit2}, we show a series of lemmas.
We define random variables $W\nn$ and $W$ as size-biased selections
from $N\nn/n$ and $\cV$. More precisely, conditionally on $N\nn$, we select
an index $I$ with distribution $\P(I=i\mid N\nn)=N_i\nn/n$, and then define
\begin{align}\label{W}
  W\nn:=
  \begin{cases}
    \xfrac{N_I\nn}{n},&I\ge1,\\
1,& I=0.
  \end{cases}
\end{align}
(The special definition in the case $I=0$, 
which has  probability $O(1/n)$ only, will
be convenient below.)
Similarly, conditionally on $\cV$ we select $I$ with $\P(I=i\mid\cV)=V_i$,
and then take
\begin{align}\label{W2}
  W:=V_I.
\end{align}
It follows from \ref{ST2} that
\begin{align}\label{W3}
  W\nn\dto W.
\end{align}
Note that 
\begin{align}\label{wlogn}
  \E \bigpar{-\log(W\nn)} 
=\E \sumib \frac{N\nn_i}{n}\Bigpar{-\log\frac{N\nn_i}{n}}
\end{align}
and, by \eqref{chi},
\begin{align}\label{wlog}
  \E \bigpar{-\log W} 
=\E \sumib V_i\bigpar{-\log V_i}
=\chi.
\end{align}

\begin{lemma}\label{LS0}
We may couple $-\log W\nn$ with a copy $\zeta\nn$ of $\zeta:=-\log W$ such
that
\begin{align}
  \E|\zeta\nn +\log W\nn|\to0
\qquad \text{as \ntoo}.
\end{align}
\end{lemma}
\begin{proof}
  By \eqref{W3}, we have 
\begin{align}
\label{ask}    
\log W\nn \dto \log W=-\zeta.
\end{align}
By the Skorohod coupling theorem \cite[Theorem~4.30]{Kallenberg},
we may assume that \eqref{ask} holds \as, and thus
\begin{align}\label{embla}
\log W\nn + \zeta\asto0.
\end{align}
Furthermore,
\begin{align}\label{wlogn2}
  \E\bigsqpar{ \bigpar{\log W\nn}^2}
=\E \sumib \frac{N\nn_i}{n}\Bigpar{-\log\frac{N\nn_i}{n}}^2 \le C,
\end{align}
since $x\log^2 x$ is bounded on \oi, and similarly
$\E\zeta^2=\E[ \bigpar{\log W}^2] \le C$.
Hence the sequence 
$\E\xpar{\log W\nn+\zeta}^2$ is uniformly bounded, and thus the sequence
$\log W\nn+\zeta$ is uniformly integrable
\cite[Theorem 5.4.2]{Gut}.
Consequently, \eqref{embla} implies
$\E|\log W\nn+\zeta|\to0$
\cite[Theorem 5.5.2]{Gut}.
\end{proof}

Let $\hN_v$ be the number of balls received by vertex $v\in\Tb$.
Thus 
\begin{align}\label{ctn2}
\ctn=\set{v\in\Tb:\hN_v\ge1}.  
\end{align}

\begin{lemma}\label{LS1}
  \begin{thmenumerate}
  \item 
There exists a constant $ C$ such that, for all $n$,
\begin{align}\label{ls1a}
  \E |\ctn| \le C n
\end{align}
and,
 more generally, for any $K$,
\begin{align}\label{ls1c}
  \E\bigabs{\set{v\in\ctn:\hN_v\ge K}} \le C n/K.
\end{align}
Furthermore, 
\begin{align}\label{ls1whp}
  |\ctn| \le C n
\qquad\text{\whp}
\end{align}

\item  Deterministically, 
  \begin{align}\label{ls1b}
    |\ctn|\ge cn.
  \end{align}
\end{thmenumerate}
\end{lemma}

\begin{proof}
  First, \eqref{ls1b} follows immediately from the fact that by \ref{ST1},
  no vertex contains more than $C_0$ balls when the construction is finished;
  hence there are at least $n/C_0$ vertices containing balls.

For \eqref{ls1a}, recall \eqref{st3c}, and assume as we may that $\gd<1/2$.
Let $r:=1/(1-\gd)<2$, and let $X_k$ be the number of vertices $v$ such that
$\hN_v\in[r^k,r^{k+1})$.
For a given $k\ge0$, generate the tree as usual, but stop at every vertex
$v$ that 
receives $\hN_v<r^{k+1}$ balls, and colour these vertices pink. 
If a pink vertex $v$ has $\hN_v\ge r^k$, recolour it red.
Since the red vertices receive disjoint sets
of balls, the number $R_k$ of them is at most $n/r^{k}$.
Condition on the set of red vertices and the numbers of balls in them, and
continue the construction of the tree. Since $r<2$, each red vertex has at
most one child $w$ with $\hN_w\ge r^k$, and by \eqref{st3c}, with probability
at least $c$ it has none. Continuing, we see that for each red vertex, the
number of descendants $w$ with $\hN_w\ge r^k$ is dominated by a geometric
distribution, and thus the expected number of such descendants is $O(1)$.
Consequently, $  \E \bigpar{X_k\mid R_k}\le C R_k$, and thus
\begin{align}
  \E X_k \le C \E R_k \le C\frac{n}{r^{k}}.
\end{align}
This yields
\begin{align}
  \E |\ctn| = \E \sumko X_k 
\le\sumko C\frac{n}{r^{k}}
= C n,
\end{align}
which is \eqref{ls1a}.

We obtain \eqref{ls1c} in the same way, summing only over $k$ with $r^{k+1}>K$.

Finally, the argument above shows that $X_k$ is stochastically dominated by
a sum 
of $\floor{n/r^k}$ independent copies of a geometric random variable
$\zeta$.
Furthermore, we may choose these to be independent also for different $k$. (The
red vertices for different $k$ are not independent, but the stochastic upper
bound that we use holds also conditioned on events for larger $k$.)
Hence,
\begin{align}
  |\ctn|\le \sum_{i=1}^{m_n}\zeta_i,
\end{align}
where $\zeta_i\in\Ge(p)$ are \iid{} with some fixed $0<p<1$,
and $m_n:=\sumko \floor{n/r^k} \le Cn$. Hence, \eqref{ls1whp} follows by the
law of large numbers.
\end{proof}

\begin{lemma}\label{LSball}
  Let $\chi>0$ be given by \eqref{chi}.
If $D_n$ is the depth of a random ball in $\ctn$, then
\begin{align}
  \frac{D_n}{\log n}\pto \frac{1}{\chi}.
\end{align}
\end{lemma}

\begin{proof}
  Consider a random ball, and suppose that it follows a path 
$v_0=o,  v_1,\dots,v_D$, ending up at a vertex $v_D$ of depth $D=D_n$.
For completeness, define $v_j:=v_D$ for $j>D$.
Let, for $k\ge0$,
\begin{align}\label{Yk}
  Y_k:=\log \hN_{v_{k-1}} - \log \hN_{v_{k}}
=-\log \frac{\hN_{v_{k}} }{ \hN_{v_{k-1}}}\ge0,
\qquad k\ge0.
\end{align}
For $k\ge0$, 
let $\cF_k$ be the \gsf{} generated by $\cN\nnx{m,v}$ with $m\ge1$ and
$d(v,o)<k$, together with $v_j$ for $j\le k$.
Then $\hN_{v_k}$ and $Y_k$ are $\cF_k$-measurable, and so is the event 
$\set{D\ge k}=\set{v_k\neq v_{k-1}}$.
Conditioned on $\cF_k$, and assuming $\hN_{v_k}=m$ and $D\ge k$,
$Y_{k+1}$ has by the definitions \eqref{Yk} and \eqref{W} 
the same distribution as $-\log W\nnx m$.

Let $\gl>0$ be fixed and let $\eps>0$.
By \refL{LS0},
there exists $B=B_\eps$ 
such that if $m\ge B$, then we can couple $-\log W\nnx m$ with
$\zeta:=-\log W$ such that
$\E|\zeta+\log W\nnx m|<\eps$.

Let $L=\floor{\gl\log n}$
and
define the stopping time $\tau$ as the smallest $k$ such that one of the
following occurs.
\begin{alphenumerate}
  
\item\label{tau1}
 $k\ge L$,

\item\label{tau2}
 $\hN_{v_k}\le B$,

\item\label{tau3} 
$k>D$, and thus the ball has come to rest.
\end{alphenumerate}

By the comments just made, 
we can couple the sequence $(Y_k)_k$ with an \iid{} sequence
$(\zeta_k)\xoo$  with $\zeta_k\eqd -\log W$ such that
on the event $\set{\tau>k}\in\cF_k$,
\begin{align}\label{glada}
\E\bigpar{|Y_{k+1}-\zeta_{k+1}|\mid \cF_k}<\eps.
\end{align}
Let, recalling \eqref{Yk},
\begin{align}\label{falk}
  X:=\sum_{k=1}^\tau\bigpar{Y_k-\zeta_k}=
\log n-\log \hN_{v_\tau} -\sum_{k=1}^\tau\zeta_k
.\end{align}
Then, \eqref{glada} implies
\begin{align}\label{chil}
\E|X|
\le 
\sum_{k=1}^{L} \E\bigabs{\bigpar{Y_k-\zeta_k}\indic{k\le \tau}}
\le L\eps.
\end{align}

Let $\cEm$ be the event $\set{D<L-1}$, and let
$\cEm':=\cEm\cap\set{\hN_{v_\tau}>B}$
and
$\cEm'':=\cEm\cap\set{\hN_{v_\tau}\le B}$.
First, if $\cEm'$ occurs, then in the definition of $\tau$, neither
\ref{tau1} 
nor \ref{tau2} may occur. 
(If $\tau\ge L$, then already $\tau-1>D$, a contradiction.) 
Hence, \ref{tau3} occurs, and thus the ball has $v_\tau$ as its final position. 
Let $\cS:=\set{v\in \Tb:\hN_v>B}$. By the definition
of $\cEm'$, we have $v_\tau\in \cS$, and thus the ball ends up in the set
$\cS$.
By \ref{ST1}, there are at most $C_0|\cS|$ such balls, and thus the
conditional probability given $\ctn$ 
that our random ball is one of them is $\le C_0|\cS|/n$.
Hence, by \eqref{ls1c},
\begin{align}\label{erik}
  P(\cEm') \le \frac{\E(C_0|\cS|)}{n} \le \frac{\CC}{B}.
\end{align}
We may increase $B$ if necessary so that $B>\CCx/\eps$, and thus
$\P(\cEm')<\eps$.

On the other hand, if $\cEm''$ holds, then, by \eqref{falk}, 
\begin{align}\label{akela}
  X &\ge \log n -\log B -\sum_{k=1}^\tau\zeta_k
\ge \log n -\log B -\sum_{k=1}^L\zeta_k.
\end{align}
 Thus, by the law of large numbers, recalling that $\E\zeta_k=\chi$ by
 \eqref{wlog}, on the event $\cEm''$,  \whp{}
 \begin{align}\label{lom}
   X \ge \log n -\log B - L(\chi+\eps)
\ge \bigpar{1 -\gl(\chi+\eps)-\eps}\log n.
 \end{align}
If $\gl<1/\chi$, and $\eps$ is so small that $\gl(\chi+\eps)+\eps<1$, 
\eqref{lom}, \eqref{chil} and Markov's inequality yield
\begin{align}
  \P(\cEm'')\le \frac{L\eps}{\bigpar{1 -\gl(\chi+\eps)-\eps}\log n}+o(1)
\le \frac{\gl\eps}{{1 -\gl(\chi+\eps)-\eps}}+o(1).
\end{align}
Hence,
\begin{align}
  \P(D<L-1)=\P(\cEm)=\P(\cEm')+\P(\cEm'')
\le \eps + \eps \frac{\gl}{{1 -\gl(\chi+\eps)-\eps}}+o(1).
\end{align}
Letting $\eps\to0$, we see that $\P(D\le \gl \log n-2)\le \P(D<L-1)\to0$.
In other words, for any $\gl<1/\chi$, 
\begin{align}\label{bill}
D> \gl \log n -2  
\qquad\text{\whp}
\end{align}

For the other side, 
assume $\gl>1/\chi$, and 
let
$\cEp$ be the event $\set{D \ge L}$. Let
$\cEp':=\set{\tau=L}$
and
$\cEp'':=\cEp\cap\set{\tau<L}$.

The law of large numbers and \eqref{falk} imply that on the event $\cEp'$, \whp,
\begin{align}
  X \le \log n -\sum_{k=1}^L \zeta_k \le \log n -(\gl\chi-\eps)\log n
= - (\gl\chi-1-\eps)\log n
.\end{align}
Hence, if  
$\eps$ is small enough, \eqref{chil} and
Markov's inequality yield
\begin{align}\label{kork}
  \P(\cEp') \le \frac{L\eps}{(\gl\chi-1-\eps)\log n} + o(1)
\le \frac{\gl\eps}{\gl\chi-1-\eps}+o(1).
\end{align}

If $\cEp''$ holds, then \ref{tau1} and \ref{tau3} cannot hold, and thus
$\hN_{v_\tau}\le B$. Hence our chosen ball belongs to a subtree rooted at
$v_\tau$ with at most $B$ balls. Conditioned on $\hN_{v_\tau}=m$, this
subtree is a copy of $\cT_m$, and since the finitely many random trees
$\cT_m$, $1\le m\le B$, all are \as{}
finite and thus have finite (random) heights $H(\cT_m)$, there exists a
constant $\CCname\CCgran$ such that 
\begin{align}\label{gran}
  \P\bigpar{H(\cT_m)>\CCgran}\le \eps,\qquad m=1,\dots,B.
\end{align}
It follows that conditioned on $\cE''$, 
\begin{align}\label{ek}
  \E\bigpar{D> L+\CCx\mid\cEp''}
\le   \E\bigpar{D> \tau+\CCx\mid\cEp''}
\le \eps.
\end{align}
Finally, combining \eqref{kork} and \eqref{ek} we obtain
\begin{align}
  \P\bigpar{D>L+\CCx}
&\le \P \bigpar{\cE'}+ \P\bigpar{D>L+\CCx \text{ and }\cE''}
\notag\\&
\le \frac{\gl\eps}{\gl\chi-1-\eps}+\eps+o(1).
\end{align}
The constant $\CCx$ may depend on $\eps$, but it follows that
for large $n$,
\begin{align}
  \P\bigpar{D>(\gl+\eps)\log n}
\le
  \P\bigpar{D>L+\CCx}
\le \frac{\gl}{\gl\chi-1-\eps}\eps+\eps+o(1).
\end{align}
Since $\eps$ can be arbitrarily small, this shows that
for any $\gl>1/\chi$ and $\gd>0$,
\begin{align}\label{bull}
D\le(\gl+\gd)\log n
\qquad\text{\whp},  
\end{align}
which  together with \eqref{bill} completes the proof.
\end{proof}

We transfer this result from balls to vertices.

\begin{lemma}\label{Lsplit1}
Let $\ctn$ and $\xi\nn_i$ be as above. Then
\begin{align}\label{lsplit1}
  \frac{\dx(\xi\nn_1,o)}{\log n}\pto\frac{1}{\chi}.
\end{align}
\end{lemma}

\begin{proof}
Again, let $L:=\floor{\gl\log n}$ for a fixed $\gl>0$.
 Let $\ZZV_k$  be the set of vertices of $\ctn$ with depth $k$,
and $\ZZB_k$ the set of balls with depth $k$;
define $\ZZV_{\le k}, \ZZV_{\ge k},\ZZB_{\le k}, \ZZB_{\ge k}$ analogously.

First,
let $\gl>1/\chi$.
Let $U_L$ be the set of all vertices of $\Tb$ with depth $L$.
For any $v\in \Tb$, conditioned on $\hN_v$, the fringe subtree $\ctn^v$
has the same distribution as $\cT_m$ with $m=\hN_v$.
Consequently, \refL{LS1} shows that
\begin{align}\label{ola}
  \E\bigpar{|\ctn^v|\mid\hN_v} \le C\hN_v
\end{align}
and thus
\begin{align}\label{olb}
  \E{|\ctn^v|} \le C\E\hN_v.
\end{align}

Since $\ZZV_{\ge L}$ is the union of the fringe trees $\ctn^v$ for $v\in U_L$, 
and $\ZZB_{\ge L}$ is the set of all balls that reach some vertex in $U_L$,
it follows from \eqref{olb}
that
\begin{align}\label{olc}
\E |\ZZV_{\ge L}| = \E\sum_{v\in U_L}|\ctn^v|
\le C\sum_{v\in U_L}\E \hN_v
= C\E\sum_{v\in U_L} \hN_v
=C\E|\ZZB_{\ge L}|
.\end{align}
However, we have by \refL{LSball},
\begin{align}\label{old}
  \E |\ZZB_{\ge L}| = n\P(D_n\ge L) = o(n).
\end{align}
Combining \eqref{olc} and \eqref{old} yields, recalling \eqref{ls1b}, 
\begin{align}\label{olf}
  \P\bigpar{\xi\nn_i \ge L}
= \E \frac{|\ZZV_{\ge L}|}{|\ctn|}
\le  C\E \frac{|\ZZV_{\ge L}|}{n} 
\le  C \frac{\E|\ZZB_{\ge L}|}{n} 
= o(1).
\end{align}

In the opposite direction,
let $\gl<1/\chi$.
Let $\eps>0$ and let $B$ be a large number.
We split $\ZZVL$ into the two sets
$\ZZVL':=\set{v\in \ZZV_{\le L}:\hN_v> B}$ and 
$\ZZVL'':=\set{v\in \ZZV_{\le L}:\hN_v\le  B}$.
By \eqref{ls1c}, we may choose $B$ so large that 
\begin{align}\label{noa}
\E|\ZZVL'| \le \eps n.
\end{align}

To treat $\ZZVL''$,
we now stop the construction of $\ctn$ at each vertex $v$ with
$\hN_v\le B$. If such a vertex also has depth $\le L$, we colour it green.
Let $\cG$ be the set of all green vertices. 
Then the set $\ZZV_{\le L}''$ is included in the union of the fringe trees
$\ctn^v$ 
for $v\in \cG$. Furthermore, 
conditioned on the set $\cG$ and $(\hN_v)_{v\in\cG}$,
each $\ctn^v$ (for $v\in\cG$)
has the same distribution as $\cT_m$ for $m=\hN_v$.
Thus, by \refL{LS1}, 
\begin{align}\label{olk}
\E\bigpar{|\ZZVL''|\mid\cG|, (\hN_v)_{v\in\cG}}
&\le \sum_{v\in\cG}\E\bigpar{ |\ctn^v|\mid \cG,(\hN_v)_{v\in\cG}}
\le \sum_{v\in\cG}C\hN_v
\notag\\&
\le CB|\cG| 
= C|\cG|
.\end{align}
Consequently, 
\begin{align}\label{oll}
  \E|\ZZVL''| \le C\E|\cG|.
\end{align}
Next, let again $\CCgran$ be such that \eqref{gran} holds, with $\eps$
replaced by 
$1/2$.
Then, still conditioned on $\cG$ and $(\hN_v)_{v\in\cG}$,
\eqref{gran} shows that each fringe tree $\ctn^v$ (for $v\in\cG$)
with probability $\ge1/2$ has height $\le \CCgran$;
if this happens, $\ctn^v$ has in particular at least
one ball of depth $\le \CCgran$ in the fringe tree, and thus depth 
$\le L+\CCgran$ in $\ctn$.
Hence,
\begin{align}
  \E\bigpar{|\ZZB_{\le L+\CCgran}|\mid\cG} \ge \tfrac12 |\cG|.
\end{align}
Together with \eqref{oll}, this yields
\begin{align}\label{olm}
\E|\ZZVL''| \le C\E|\cG|\le C\E|\ZZB_{\le L+\CCgran}|
\end{align}
and then \refL{LSball} implies
\begin{align}\label{olq}
\E|\ZZVL''| \le C\E|\ZZB_{\le L+\CCgran}|
= Cn\P\bigpar{D_n\le L+\CCgran}=o(n).
\end{align}
 
By \eqref{noa} and \eqref{olq}, we have for large $n$
\begin{align}
  \E|\ZZVL|
=
  \E|\ZZVL'|+  \E|\ZZVL''|\le 2\eps n.
\end{align}
Thus, 
$  \E|\ZZVL|=o(n)$, and, similarly to \eqref{olf},
\begin{align}\label{olg}
  \P\bigpar{\xi\nn_i \le L}
= \E \frac{|\ZZV_{\le L}|}{|\ctn|}
\le  C\E \frac{|\ZZV_{\le L}|}{n} 
= o(1).
\end{align}
This completes the proof together with \eqref{olf}.
\end{proof}

\begin{lemma}
  \label{LSB}
With notations as above,
\begin{align}\label{lsplit2x}
  \frac{\dx(\xi\nn_1\bmin\xi\nn_2,o)}{\log n}\pto0.
\end{align}
\end{lemma}

\begin{proof}
There is a standard identification of the vertices of $\Tb$ with finite
strings $i_1\dotsm i_k$ with 
$k\ge0$ and $i_j\in\set{1,\dots,b}$. 
If $v=i_1\dotsm i_k\in \Tb$, let $v_j:=i_1\dotsm i_j$, $j\le k$,
and define
\begin{align}\label{VV}
  \VV_v:=\prod _{j=0}^{k-1} V\nnx{v_j}_{i_{j+1}}.
\end{align}
Then 
\cite[Lemma 2]{Broutin+++},
by \ref{ST2} and induction over $k$, as \ntoo,
\begin{align}\label{no1}
  \hN_v/n \pto \VV_v,
\qquad v\in \Tb.
\end{align}
Furthermore, it follows from \eqref{ls1whp} that for any $\eps>0$
and any fixed $v\in \Tb$, \whp
\begin{align}\label{no2}
  |\ctn^{v}|\le C \hN_v+\op(n).
\end{align}
(The term $\op(n)$ takes care of the possibility that $\hN_v$ is small; we
have not excluded the case $\VV_v=0$.)
By \eqref{no1} and \eqref{no2}, 
\begin{align}
  |\ctn^v|\le (C\VV_v+\op(1)) n
\end{align}
and thus, for any fixed $K$,
\begin{align}\label{brage}
\sum_{v\in U_K}    |\ctn^v|^2\le \sum_{v\in U_K}(C\VV_v+\op(1))^2 n^2
.\end{align}
Since $d(\xi\nn_1\bmin\xi\nn_2,o)\ge K)$
if and only if the two vertices $\xi\nn_1$ and $\xi\nn_2$ are in the same
subtree $\ctn^v$ for some $v\in U_K$, it follows from \eqref{brage} that, 
using also \eqref{ls1b} and $\sum_{v\in U_K}\VV_v=1$,
\begin{align}
\P\bigpar{ d\xpar{\xi\nn_1\bmin\xi\nn_2,o)\ge K\mid\ctn}}&
=\frac{1}{|\ctn|^2}\sum_{v\in U_K}    |\ctn^v|^2
\notag\\&
\le C\sum_{v\in U_K}(C\VV_v+\op(1))^2
\notag\\&
= C\sum_{v\in U_K}\VV_v^2+\op(1).
\end{align}
Since the probability on the \lhs{} is bounded by 1, 
we may assume that so is the term $\op(1)$ on the \rhs, 
and thus we may take the expectation and use dominated
convergence to conclude
\begin{align}\label{no3}
\P\bigpar{ d\xpar{\xi\nn_1\bmin\xi\nn_2,o}\ge K}
&
\le C\E\sum_{v\in U_K}\VV_v^2+o(1).
\end{align}
Furthermore, by the definition \eqref{VV} and independence,
\begin{align}\label{no4}
  \E\sum_{v\in U_K}\VV_v^2
=\sum_{i_1,\dots,i_K}\prod_{j=1}^K\E V_{i_j}^2
= \Bigpar{ \sum_{i=1}^b\E V_i^2}^K.
\end{align}
Since $\sum_i V_i^2\le\sum_i V_i= 1$, and strict inequality holds with positive
probability,
\begin{align}\label{no5}
   \sum_{i=1}^b\E V_i^2 =   \E \sum_{i=1}^bV_i^2<1.
\end{align}
Hence, given any $\eps>0$, we can find $K$ such that \eqref{no3} yields
\begin{align}\label{no7}
\P\bigpar{ d\xpar{\xi\nn_1\bmin\xi\nn_2,o)\ge K}}
\le C \frac{\eps}{2C}+o(1) <\eps
\end{align}
for large $n$.
In particular, \eqref{lsplit2x} follows. (In fact, we have proved that the
sequence $ d\xpar{\xi\nn_1\bmin\xi\nn_2,o}$ of random variables is tight.)
\end{proof}

\begin{proof}[Proof of \refT{Tsplit2}]
\refT{Tsplit2}
follows from \refLs{Lsplit1} and   \ref{LSB}
by \refT{Tlog}.
\end{proof}

\begin{remark}\label{RRT}
The random recursive tree
and preferential attachment trees are not split trees in the sense above,
since 
degrees are unbounded.
Nevertheless, if the definition above is generalized to allow $b=\infty$,
they too can be regarded as split trees, see \cite{SJ320}.
We conjecture that under suitable conditions, \refT{Tsplit2} extends to the
case $b=\infty$, but we have not pursued this.
(Random recursive trees and preferential attachment trees can be handled by
\refT{TCMJ} below instead.)
\end{remark}

\section{Crump--Mode--Jagers branching process trees}
\label{SCMJ}

A \emph{Crump--Mode--Jagers} (CMJ) branching process 
(see \eg \cite{Jagers})
is a continuous time
process,
where each individual gives birth to a (generally random) number of children
at arbitrary random times; 
the times a single individual gets children are thus
described by a point process $\Xi$ on $\ooo$. 
All individuals have independent and
identically distributed such point processes. 
We start with a single individual, born at time 0; we also suppose that the
CMJ process is supercritical and that it never dies out; hence its size
\as{} grows to $\infty$.

The family tree of the CMJ process is a growing random tree $\ctt$,
$t\ge0$, where the vertices are all individuals born up to time $t$.
We stop the tree at the stopping time $\tau(n)$ where the tree
first reaches 
$n$ vertices. Then (provided births \as{} occur at distinct times)
$\ctn:=\ctx{\tau(n)}$ is a random tree with fixed size
$|\ctn|=n$. More generally, $\tau(n)$ may be defined as the first time 
the total weight reaches $n$, where each individual has a weight
given by some ``characteristic'' $\psi$; 
see \cite{SJ306} for details. (For example, for an $m$-ary
search tree, $\psi$ counts the balls, and we stop when there are $n$ balls; 
\cf{} the split trees in \refS{Ssplit}.)

Many examples of such CMJ trees are discussed in the survey 
\cite[Sections 6--8]{SJ306};
these include for example binary search trees and $m$-ary search trees (also
covered by \refS{Ssplit}), and random recursive trees and preferential
attachment trees. We give in \refT{TCMJ} a general result for such trees.
For example, this applies to \refEs{EBST}--\ref{Ebinc}.

The point process $\Xi$ can informally be regarded as the random set
$\set{\xix_i}_{i=1}^N$ of the times of births $\xix_i$ of the children of the
root, where the number of children $N\in\set{0,1,\dots,\infty}$ in general
is random. 
(We use the notation $\xix_i$ to avoid confusion with the random vertices
$\xi\nn_i$.)
Formally, $\Xi$ is defined as the random measure
$\sum_{i=1}^N\gd_{\xix_i}$. Let $\mu:=\E\Xi$ denote the intensity measure of $\Xi$.

We define the Laplace transform of the measure $\mu$ on $\ooo$ by
\begin{equation}\label{Lm}
  \hmu(\gth) = \intoo e^{-\gth t}\mu(\ddx t)
=\E\intoo\sumiN e^{-\gth \xix_i}
,\qquad -\infty<\gth<\infty.
\end{equation}

As in \cite{SJ306}, we make the following assumptions;
see further \cite{SJ306}.
\begin{PXenumerate}{A}
\item \label{BPfirst}
$\mu\set0=\E\Xi\set0<1$. 
(This rules out a rather trivial case with explosions
already at the start.
In all examples in \cite{SJ306}, $\mu\set0=0$.)
\item \label{BPnonlattice}
$\mu$ is not concentrated on any lattice $h\bbZ$, $h>0$.
(This is for convenience only.)  
\item \label{BPsuper}
$\E N>1$. 
(This is known as  the \emph{supercritical} case.)
For simplicity, we further assume that $N\ge1$ a.s.
(In this case, every individual has at least one child, so the process never
dies out and $|\ctx\infty|=\infty$.)
\item \label{BPmalthus}
There exists a real number $\ga>0$ (the \emph{Malthusian parameter}) such that
$\hmu(\ga)=1$, \ie, 
\begin{equation}\label{malthus}
\intoo e^{-\ga t}\mu(\ddx t) =1.  
\end{equation}

\item \label{BPmub}
$\hmu(\gth)<\infty$ for some $\gth<\ga$.

\renewcommand{\labelenumi}{{\upshape{(A\arabic{enumi}$\psi$)}}}
\item \label{BPlast}
(Only needed if the stopping time $\tau(n)$ is defined using a
weight $\psi$. 
Thus void in the case that $\ctn$ always has $n$ vertices.)
The random variable $\sup_t \bigpar{e^{-\gth t}\psi(t)}$ has finite
expectation for 
some $\gth<\ga$.
\setcounter{oldenumi}{\value{enumi}}
\end{PXenumerate}

We assume also the following  technical condition.
(We conjecture that this is not necessary, but we use it in our proof.)
Define the random variable
\begin{align}\label{xia1}
  \Xia:=\sumiN e^{-\ga\xix_i}
\end{align}
  and note that \eqref{malthus} is equivalent to
\begin{align}\label{xia2}
\E\Xia=1.    
\end{align}
We assume a weak moment condition. 
\begin{PXenumerate}{A}
\setcounter{enumi}{\value{oldenumi}}
\item \label{BPLlogL}
We have
$
\E \bigsqpar{\Xia \log\Xia} <\infty.    
$
\end{PXenumerate}

\begin{remark}\label{Rlogl}
Note that \ref{BPLlogL} trivially holds if the outdegrees in $\ctn$ are
bounded, so $N\le C$ \as{} for some $C\le\infty$. It is also easily seen
that
\ref{BPLlogL} holds, as a consequence of
the stronger
$\E \bigsqpar{\Xia^2} <\infty$,
for random recursive trees and the linear preferential attachment trees
in \cite[Section 6]{SJ306}.
\end{remark}

We let for convenience $Z_t:=|\ctt|$, and similarly $Z_t^v:=|\ctt^v|$ for
fringe trees. We also define
\begin{equation}\label{el}
\gb:=\intoo te^{-\ga t}\mu(\ddx t) <\infty
.\end{equation}

\begin{remark}\label{RW}
\citet{Nerman} showed that
under the assumptions \ref{BPfirst}--\ref{BPlast} above, 
there exists a random
variable $W$ such that, as \ttoo,
\begin{equation}\label{olle1}
  e^{-\ga t}Z_t \asto W.
\end{equation}
If furthermore \ref{BPLlogL} holds, then
$W>0$ \as{} and
\begin{equation}  \label{EW}
\E W=(\ga\gb)\qw.
\end{equation}
However, if \ref{BPLlogL} fails, then $W=0$ \as.
See also \cite{Doney}.
\end{remark}

\begin{theorem}
  \label{TCMJ}
Assume
\ref{BPfirst}--\ref{BPlast} and \ref{BPLlogL}.
Then
\begin{align}\label{bpa}
  \frac{1}{\log n}\ctn \pto\OO_{1/(\ga\gb)}.
\end{align}
\end{theorem}

\begin{proof}
  It is shown in \cite[Theorem 13.61]{SJ306-arxiv},
using results by \citet{Nerman} and \citet{Biggins95,Biggins97},
that \eqref{olof} holds with $a:=1/(\ga\gb)$.
Hence, by \refT{Tlog}, it remains only to verify \eqref{twin}.
We argue similarly as for \refL{LSB}.

We regard $\ctt$ as a subtree of the infinite tree $\Too$,
where the vertices are all finite strings $i_1\dotsm i_k$ of natural
numbers $i_j\in \bbN$, with $0\le k<\infty$; thus the children of $v$ are
$vi$, for $i=1,\dots$, in this order.
Note that the length of the string labelling $v\in\Too$ equals $d(v,o)$; we
denote this length by $|v|$.

For  a vertex $v\in\Too$, let $\bb_v$ be the time that $v$ is born
in our CMJ branching process; if $v$ never appears, 
then $\bb_v:=\infty$.

The fringe tree $\ctt^v$ (defined as $\emptyset$ if $\bb_v>t$ so $v\notin\ctt$)
is from the time $\bb_v$ on a copy of the entire branching process tree, and
thus \eqref{olle1} implies that for every $v$ with $\bb_v<\infty$,
\begin{align}\label{bp2}
  e^{-\ga(t-\bb_v)}Z_t^v\asto W_v,\qquad\ttoo,
\end{align}
where $W_v\eqd W$ is independent of $\bb_v$. Thus
\begin{align}\label{bp3}
  e^{-\ga t}Z_t^v\asto e^{-\ga\bb_v}W_v, \qquad\ttoo,
\end{align}
which holds trivially also for $\bb_v=\infty$ (with
$e^{-\infty}=0$). Consequently, 
\begin{align}\label{bp4}
\frac{Z_t^v}{Z_t}=\frac{e^{-\ga t}Z_t^v}{e^{-\ga t}Z_t}
\asto  \frac{e^{-\ga\bb_v}W_v}W=: Y_v, \qquad\ttoo.
\end{align}

Consider first the children of the root; these are labelled with $i\in\bbN$.
Since $Z_t=1+\sum_i Z_t^i$, we have by \eqref{bp4} and (the elementary)
Fatou's lemma for sums, \as,
\begin{align}\label{bpe}
  \sum_i Y_i
=\sum_i \liminf_\ttoo \frac{Z_t^i}{Z_t}
\le \liminf_\ttoo \sum_i \frac{Z_t^i}{Z_t}
=1.
\end{align}
Equivalently, by \eqref{bp4},
\begin{align}\label{bc1}
\sum_i{e^{-\ga\bb_i}W_i}\le W, \qquad\textas
\end{align}
Furthermore, by  \eqref{xia1}--\eqref{xia2}, 
noting that $\bb_i=\xix_i$,
\begin{align}\label{bpf}
%
\E \sum_i {e^{-\ga\bb_i}W_i }
&=
 \sum_i \E\bigsqpar{e^{-\ga\bb_i}W_i }
= \sum_i \E\bigsqpar{e^{-\ga\bb_i}}\E \sqpar{W_i}
\notag\\&
= \E [W]\sum_i \E{e^{-\ga\bb_i}}
= \E [W]\E\sum_i {e^{-\ga\bb_i}}
=\E [W]\E\Xia
\notag\\&
=\E W.
\end{align}
By \eqref{EW}, 
$\E W<\infty$ and thus
\eqref{bc1} and \eqref{bpf} imply
\begin{align}\label{bc2}
\sum_i{e^{-\ga\bb_i}W_i}= W, \qquad\textas
\end{align}
Equivalently, there is \as{} equality in \eqref{bpe}.

Let $v\in\Too$ and apply \eqref{bc2} to the fringe tree $\ctt^v$, 
again regarded as a copy of the original branching process; this shows that if
$b_v<\infty$, then
\begin{align}\label{bc3}
\sum_i{e^{-\ga(\bb_{vi}-\bb_v)}W_{vi}}= W_v, \qquad\textas
\end{align}
and thus
\begin{align}\label{bc4}
\sum_i{e^{-\ga\bb_{vi}}W_{vi}}= e^{-\ga\bb_v}W_v, \qquad\textas,
\end{align}
where \eqref{bc4} trivially holds also if $b_v=\infty$.

By \eqref{bc4} and induction we conclude that for every $k\ge0$,
\begin{align}\label{bc5}
  \sum_{|v|=k}{e^{-\ga\bb_{v}}W_{v}}= W, \qquad\textas
\end{align}
Equivalently, by the definition \eqref{bp4} again,
\begin{align}\label{bpm}
  \sum_{|v|=k}{Y_{v}}= 1, \qquad\textas
\end{align}

Next, fix an integer $k$.
Two vertices $v$ and $w$ of $\ctt$ have $d(v\bmin w,o)\ge k$ if and only if
they belong to the same subtree $\ctt^v$ for some $v$ with $|v|=k$.
Thus, if $\xi\nnx t_j$ are \iid uniformly random vertices in $\ctt$,
\begin{align}\label{bpr}
  \P\bigpar{d(\xi\nnx t_1\bmin\xi\nnx t_2,o)\ge k\mid\ctt}
= \sum_{|v|=k}\Bigparfrac{Z_t^v}{Z_t}^2
.\end{align}
By \eqref{bpm} and Fatou's lemma as in \eqref{bpe}, \as,
\begin{align}\label{bps}
  1=\sum_{|v|=k}\YY_v
\le \liminf_\ttoo\sum_{|v|=k}\frac{Z_t^v}{Z_t}
\le \limsup_\ttoo\sum_{|v|=k}\frac{Z_t^v}{Z_t}
\le 1,
\end{align}
and thus
\begin{align}\label{bpt}
\sum_{|v|=k}\frac{Z_t^v}{Z_t}\asto1.
\end{align}
This together with \eqref{bp4} and \eqref{bpm} implies by a standard
argument,
\cf{} again \cite[Theorem 5.6.4]{Gut},
\begin{align}\label{bpkk}
\sum_{|v|=k}\lrabs{  \frac{Z_t^v}{Z_t}- \YY_v}\asto0.
\end{align}
Hence,
\begin{align}\label{bpkkk}
\sum_{|v|=k}\lrabs{  \Bigparfrac{Z_t^v}{Z_t}^2- \YY_v^2}
\le
\sum_{|v|=k}\lrabs{  \frac{Z_t^v}{Z_t}- \YY_v}\asto0
\end{align}
and thus \eqref{bpr} implies
\begin{align}\label{bpu}
  \P\bigpar{d(\xi\nnx t_1\bmin\xi\nnx t_2,o)\ge k\mid\ctt}
\asto \sum_{|v|=k}\YY_v^2
.\end{align}
By considering the sequence of times $\tau(n)$, this shows
\begin{align}\label{bpv}
  \P\bigpar{d(\xi\nn_1\bmin\xi\nn_2,o)\ge k\mid\ctn}
\asto \sum_{|v|=k}\YY_v^2
.\end{align}
Taking the expectation yields, by dominated convergence, 
\begin{align}\label{bpw}
  \P\bigpar{d(\xi\nn_1\bmin\xi\nn_2,o)\ge k}
\to \E\sum_{|v|=k}\YY_v^2
.\end{align}

We want to show that the \rhs{} of \eqref{bpw} tends to 0 as \ktoo.
Define, for $k\ge0$,
\begin{align}\label{bc10}
  Q_k:=\sum_{|v|=k}\bigpar{e^{-\ga b_v}W_v}^2
=\sum_{|v|=k}e^{-2\ga b_v}W_v^2
= W^2 \sum_{|v|=k}Y_v^2
.\end{align}
By \eqref{bc4}, \as,
\begin{align}\label{bc11}
  W^2=Q_0\ge Q_1\ge Q_2 \ge\dots
\end{align}
Define 
\begin{align}\label{bc12}
  Q_\infty:=\lim_{\ktoo} Q_k 
.\end{align}
Similarly, for each $i\in\bbN$ with $b_i<\infty$,
consider the fringe tree $\ctt^i$, and define 
\begin{align}\label{bc13k}
  Q_{k;i}&:=\sum_{|v|=k} e^{-2\ga(b_{iv}-b_{i})} W_{iv}^2,
\\
Q_{\infty;i}&:=\lim_\ktoo Q_{k;i}
\eqd Q_\infty
\label{bc13oo}
.\end{align}
For convenience, we define $Q_{\infty:i}$ also when $b_i=\infty$, as some 
copy of $Q_\infty$ independent of everything else.
Then, \eqref{bc13k} and \eqref{bc10} yield, for any $k\ge0$,
\begin{align}\label{bc14}
  Q_{k+1}=\sumi e^{-2\ga b_i} Q_{k;i}.
\end{align}
Letting $k\to\infty$ in \eqref{bc14}, we obtain by dominated convergence, since
$Q_{k;i}\le Q_{0;i}$ and $\sum_i e^{-2\ga b_i} Q_{0;i}=Q_1\le
Q_0=W^2<\infty$ \as,
\begin{align}\label{bc15}
  Q_{\infty}=\sumi e^{-2\ga b_i} Q_{\infty;i}
\qquad\textas
\end{align}

We claim that $Q_\infty=0$ \as.
To see this note first that \eqref{bc15} implies
\begin{align}\label{bc44}
  Q_{\infty}\qq\le\sumi e^{-\ga b_i} Q_{\infty;i}\qq
\qquad\textas,
\end{align}
with strict inequality as soon as there is more than one non-zero term in
the sum.
Moreover,
since $b_i$ and $Q_{\infty;i}$ are independent,
using \eqref{xia1}--\eqref{xia2} again,
\begin{align}\label{bc0}
  \E \sumi e^{-\ga b_i} Q_{\infty;i}\qq
&
=  \sumi\E\bigsqpar{ e^{-\ga b_i} Q_{\infty;i}\qq}
=  \sumi\E\bigsqpar{ e^{-\ga b_i}}\E\bigsqpar{ Q_{\infty;i}\qq}
\notag\\&
= \E\bigsqpar{ Q_{\infty}\qq} \sumi\E\bigsqpar{ e^{-\ga b_i}}
= \E\bigsqpar{ Q_{\infty}\qq} \E\bigsqpar{\Xia}
\notag\\&
= \E{ Q_{\infty}\qq}
.\end{align}
Furthermore, $\E Q_\infty\qq\le \E W<\infty$.
Hence, 
\eqref{bc0} implies
\begin{align}\label{bc00}
  \E\Bigpar{ \sumi e^{-\ga b_i} Q_{\infty;i}\qq
- Q_{\infty}\qq}
&
=
\E{ Q_{\infty}\qq} - \E{ Q_{\infty}\qq}
=0
,\end{align}
and thus there is equality in \eqref{bc44} \as.

Suppose that $\P(Q_\infty>0)>0$. 
Conditioned on the offspring $\Xi$ of the root, the 
fringe trees $\ctt^i$, $i \le N$, are independent copies of $\ctt$. Hence,
the events $N\ge 2$, $Q_{\infty;1}>0$ and $Q_{\infty;2}>0$ are independent 
and thus with positive probability they occur together, and then there is
strict inequality in \eqref{bc44}.
This contradiction shows that $Q_\infty=0$ \as.

Consequently,
\eqref{bc10} shows that, since  $W>0$ \as{},
\begin{align}\label{ad1}
  \sum_{|v|=k} Y_v^2 = W\qww Q_k \asto W\qww Q_\infty = 0,
\qquad \ktoo.
\end{align}
Furthermore, $  \sum_{|v|=k} Y_v^2 \le1$ by \eqref{bpm} or \eqref{bpv}.
Hence, by dominated convergence,
\begin{align}\label{ad2}
\E  \sum_{|v|=k} Y_v^2 \to 0,
\qquad \ktoo.
\end{align}

Finally, \eqref{bpw} and \eqref{ad2} show that
\begin{align}
  \lim_{\ktoo}\lim_{\ntoo}
  \P\bigpar{d(\xi\nn_1\bmin\xi\nn_2,o)\ge k}
=0,
\end{align}
which shows that the sequence of random variables 
$d(\xi\nn_1\bmin\xi\nn_2,o)$ is tight, and in particular that \eqref{twin}
holds.
\end{proof}

\section{Proof of \refT{ET1}}\label{SpfET1}

\refT{ET1} is stated in \cite[Theorem 1]{ET} for uniformly bounded rescaled
finite trees. Furthermore, \cite[Theorem 4]{ET} contains a related statement
(for measured real trees); we show that it implies \refT{ET1}.

\begin{proof}[Proof of \refT{ET1}]
This is the only place in the present paper where we use the machinery 
with ultraproducts
used
in \cite{ET} to prove the results there. 
We refer to \cite{ET} for definitions and basic properties, and will here
only give the additional arguments needed.
We
fix, as in \cite{ET}, an ultrafilter $\go$ on $\bbN$. All ultralimits and
ultraproducts are defined using $\go$.

Let $(T_n)\xoo=(T_n,d_n,\mu_n)\xoo$ be a convergent sequence of measured
real trees. Thus \eqref{d2} holds for some measures $\xnu_r\in \cP(M_r)$.

Taking $r=2$ in \refD{D2}, 
we see by \eqref{taurus} and \eqref{rhor} that, in particular, 
\begin{align}\label{greta}
  d_n(\xi_1\nn,\xi_2\nn) \dto \zeta,
\end{align}
for some random variable $\zeta$.
It follows from \eqref{greta} that the sequence of random variables 
$  d_n(\xi_1\nn,\xi_2\nn)$ is tight, \ie, that for every $\ep>0$, there
exists a constants $\Ce$
such that for every $n$
\begin{align}\label{hans}
\P\bigpar{  d_n(\xi_1\nn,\xi_2\nn)>\Ce} \le\eps.
\end{align}
Fix $\eps>0$.
By \eqref{hans} and Fubini's theorem, there exists $x_n\in T_n$ such that
\begin{align}\label{ego}
\P\bigpar{  d_n(\xi_1\nn,x_n)>\Ce} \le\eps.
\end{align}
Let $A_n:=\set{x\in T_n:d_n(x,x_n)\le\Ce}$. Then \eqref{ego} says
\begin{align}\label{bmo}
\mu_n(A_n)\ge 1-\eps.  
\end{align}

As in \cite{ET}, form the ultraproduct $\QT:=\prod_\go T_n$, and equip it
with the pseudometric $\Qd:= \lim_\go d_n$ (which may take the value $+\infty$)
and the probability measure
$\Qmu:=\prod_\go \mu_n$.
Let $\Qx:=[(x_n)_n]\in\QT$ and $\Q{A}:=\prod_\go A_n\subseteq\QT$. 
For any $\Q{y}\in \Q{A}$,
$\Q{y}=[(y_n)_n]$ for some $y_n\in T_n$ 
with $y_n\in A_n$ and thus $d_n(x_n,y_n)\le \Ce$ for every $n$;
hence
\begin{align}\label{ada}
  \Qd(\Q{x},\Q{y})=\lim_\go d_n(x_n,y_n) \le \Ce.
\end{align}
Furthermore, by \eqref{bmo},
\begin{align}\label{ida}
  \Qmu(\Q A)=\lim_\go\mu_n(A_n)\ge 1-\eps.
\end{align}
Let $\Q{X}:=B(\Qx,\infty):=\set{\Q{y}:\Qd(\Q{y},\Qx)<\infty}$.
Then \eqref{ada} shows that $\Q{A}\subset \Q{X}$, and thus \eqref{ida}
shows
\begin{align}\label{eda}
  \Qmu\xpar{\Q X}\ge \Qmu\xpar{\Q A} \ge 1-\eps.
\end{align}
(It is shown in \cite{ET} that $\Q X$ is $\Qmu$-measurable.)
Here $\Qx=\Qx(\eps)$ and $\Q X=\Q X(\eps)$ may depend on $\eps$. 
However, two infinite balls $B(\Qx_1,\infty)$ and $B(\Qx_2,\infty)$ in $\QT$
either coincide or are disjoint. (Such infinite balls are called
\emph{clusters} in \cite{ET}.)
Hence, considering only $\eps\le\frac12$, it follows from \eqref{eda} that
all $\Q X(\eps)$ coincide, and consequently form a cluster $\Q X$ with,
using \eqref{eda} again,
$\Qmu(\Q X)=1$.

This means that the sequence $(T_n,d_n,\mu_n)_n$ is 
\emph{essentially  bounded}, in the terminology of \cite{ET}.
Consequently, \cite[Theorem 4]{ET} applies, and shows that
  \begin{align}\label{yda}
\lim_\go \tau_r(T_n)=    \lim_\go \tau_r(T_n,d_n,\mu_n) = \tau_r(D),
  \end{align}
for every $r\ge1$ and some long dendron $D$ (constructed from the
ultraproduct $\QT$ in a way that we do not have to consider further).

On the other hand, we have assumed \eqref{d2}, so the sequence 
$\tau_r(T_n)=\tau_r(T_n,d_n,\mu_n)$  converges.
A convergent sequence has its limit as its ultralimit; hence \eqref{yda} and
\eqref{d2} yield $\tau_r(D)=\gl_r$. Consequently, \eqref{d2} says
\begin{align}
  \tau_r(T_n)\to \tau_r(D),
\qquad r\ge1,
\end{align}
and thus $T_n\to D$, which completes the proof.
\end{proof}

\begin{remark}
The proof shows that a tight sequence $(T_n)_n$ 
of measured real trees is essentially
bounded.
The converse does not hold, since we may let $T_n$ be arbitrary along some
subsequences without affecting the ultraproduct  and ultralimits, and thus
the property of being essentially bounded. Nevertheless, a sequence $(T_n)_n$
such that every subsequence is essentially bounded is tight
(as a consequence of \cite[Theorem 4]{ET}).
Similarly, a sequence is tight if and only if it is essentially bounded for
every ultrafilter $\go$.
\end{remark}

\section*{Acknowledgement}
I thank Cecilia Holmgren for help with references.

\newcommand\AAP{\emph{Adv. Appl. Probab.} }
\newcommand\JAP{\emph{J. Appl. Probab.} }
\newcommand\JAMS{\emph{J. \AMS} }
\newcommand\MAMS{\emph{Memoirs \AMS} }
\newcommand\PAMS{\emph{Proc. \AMS} }
\newcommand\TAMS{\emph{Trans. \AMS} }
\newcommand\AnnMS{\emph{Ann. Math. Statist.} }
\newcommand\AnnPr{\emph{Ann. Probab.} }
\newcommand\CPC{\emph{Combin. Probab. Comput.} }
\newcommand\JMAA{\emph{J. Math. Anal. Appl.} }
\newcommand\RSA{\emph{Random Structures Algorithms} }
\newcommand\DMTCS{\jour{Discr. Math. Theor. Comput. Sci.} }

\newcommand\AMS{Amer. Math. Soc.}
\newcommand\Springer{Springer-Verlag}
\newcommand\Wiley{Wiley}

\newcommand\vol{\textbf}
\newcommand\jour{\emph}
\newcommand\book{\emph}
\newcommand\inbook{\emph}
\def\no#1#2,{\unskip#2, no. #1,} 
\newcommand\toappear{\unskip, to appear}

\newcommand\arxiv[1]{\texttt{arXiv}:#1}
\newcommand\arXiv{\arxiv}

\def\nobibitem#1\par{}

\end{document}